\documentclass[12pt]{amsart}
\usepackage[margin = 1in]{geometry}
\usepackage[english]{babel}
\usepackage{amsmath}         
\usepackage{mathtools}       
\usepackage{amssymb}   
\usepackage{amsthm}
\usepackage{mathrsfs}        			                  
\usepackage{enumitem}
\usepackage[dvipsnames]{xcolor}
\usepackage{xcolor}
\usepackage{tikz}
\usepackage{tikz-cd}
\usepackage{pgfplots}
\usepackage{float}
\usepackage{subfig}
\usepackage{graphicx}
\usepackage{comment}
\usepackage{textcomp}
\usetikzlibrary{decorations.pathmorphing,positioning}
\usetikzlibrary{calc,positioning}
\usetikzlibrary{decorations.markings}

\usepackage{hyperref}
\hypersetup{
	colorlinks=true,
	linkcolor=blue,
	citecolor=blue,
	urlcolor=blue}

\newcommand{\ZZ}{\mathbb{Z}}
\newcommand{\QQ}{\mathbb{Q}}
\newcommand{\RR}{\mathbb{R}}

\newcommand{\A}{\mathcal A}

\newcommand{\F}{\mathcal F}

\renewcommand{\L}{\mathcal L}
\newcommand{\M}{\mathcal M}

\renewcommand{\P}{\mathcal P}

\newcommand{\W}{\mathcal W}

\newcommand{\Mct}{\mathcal{M}_{g,n}^{ct}}
\newcommand{\Mbar}{\overline{\mathcal{M}}_{g,n}}

\newcommand{\deltasp}{\delta_{split}}
\newcommand{\deltal}{\delta_{loop}}

\newcommand{\GK}{\text{GK}^{bl,k}_{g,n}}
\newcommand{\GKzero}{\text{GK}^{bl,0}_{g,n}}
\newcommand{\GKtwo}{\text{GK}^{bl,2}_{g,n}}

\newcommand{\GKbartwo}{\overline{\text{GK}}^{bl,2}_{g,n}}
\newcommand{\GKbarzero}{\overline{\text{GK}}^{bl,0}_{g,n}}
\newcommand{\deltairr}{\delta_{irr}}

\newcommand{\GKbarzeroBM}{(\overline{\mathrm{GK}}^{bl,0}_g)^\vee}

\tikzset{
  gkvertex/.style={circle,draw,inner sep=1pt},
  every loop/.style={draw}
}

\usetikzlibrary{calc}

\newcommand{\attachedwheel}[5]{%
  \coordinate (#1-hub) at
    ($ (#3)+({#5*cos(#4)},{#5*sin(#4)}) $);

  \pgfmathtruncatemacro{\lastvertex}{#2-1}

  \foreach \k in {0,...,\lastvertex}{
    \coordinate (#1-\k) at
      ($ (#1-hub)
       +({#5*cos(#4+180+360*\k/#2)},
         {#5*sin(#4+180+360*\k/#2)}) $);
  }

  \foreach \k in {0,...,\lastvertex}{
    \pgfmathtruncatemacro{\nextvertex}{mod(\k+1,#2)}
    \draw[thick] (#1-\k)--(#1-\nextvertex);
    \draw[thick] (#1-hub)--(#1-\k);
  }

  \foreach \k in {0,...,\lastvertex}
    \fill (#1-\k) circle (1.25pt);

  \fill (#1-hub) circle (1.25pt);
}

\tikzset{
    smallvertex/.style={circle, draw, minimum size=14pt, inner sep=0pt},
    edge/.style={thick}
}

\tikzset{doublevertex/.style={circle, draw, double, minimum size=16pt, inner sep=0pt},  edge/.style={thick}
}

\newcommand{\graphab}{%
\tikz[baseline=-0.6ex]{
    \node[smallvertex] (a) at (0,0) {$a$};
    \node[smallvertex] (b) at (1.2,0) {$b$};

    \draw[edge] (a) -- (b);

    \draw[edge] ($(a)!0.5!(b)$) ++(0,-0.15) -- ++(0,0.3);

    \draw[edge] (a) -- ++(90:0.8);
    \draw[edge] (a) -- ++(150:0.8);
    \draw[edge] (a) -- ++(-150:0.8);

    \draw[edge] (b) -- ++(30:0.8);
    \draw[edge] (b) -- ++(-30:0.8);
    \draw[edge] (b) -- ++(-90:0.8);
}}

\newcommand{\graphabsl}{%
\tikz[baseline=-0.6ex]{
    \node[circle, draw, minimum size=14pt, inner sep=0pt] (a) at (0,0) {$a$};
    \node[circle, draw, minimum size=14pt, inner sep=0pt] (bp) at (1.1,0) {$b'$};
    \node[circle, draw, minimum size=14pt, inner sep=0pt] (bpp) at (2.2,0) {$b''$};

    \draw[thick] (a) -- (bp);
    \draw[thick] (bp) -- (bpp);
    \draw[edge] ($(a)!0.5!(b)$) ++(0,-0.15) -- ++(0,0.3);

    \draw[thick] (a) -- ++(90:0.65);
    \draw[thick] (a) -- ++(150:0.65);
    \draw[thick] (a) -- ++(-150:0.65);

    \draw[thick] (bp) -- ++(-90:0.65);

    \draw[thick] (bpp) -- ++(30:0.65);
    \draw[thick] (bpp) -- ++(-30:0.65);
}}

\newcommand{\graphabsr}{%
\tikz[baseline=-0.6ex]{
    \node[circle, draw, minimum size=14pt, inner sep=0pt] (a) at (0,0) {$a'$};
    \node[circle, draw, minimum size=14pt, inner sep=0pt] (bp) at (1.1,0) {$a''$};
    \node[circle, draw, minimum size=14pt, inner sep=0pt] (bpp) at (2.2,0) {$b$};

    \draw[thick] (a) -- (bp);
    \draw[thick] (bp) -- (bpp);
    \draw[edge] ($(bp)!0.5!(bpp)$) ++(0,-0.15) -- ++(0,0.3);

    \draw[thick] (a) -- ++(90:0.65);
    \draw[thick] (a) -- ++(150:0.65);

   \draw[thick] (bp) -- ++(-150:0.65);

    \draw[edge] (bpp) -- ++(30:0.8);
    \draw[edge] (bpp) -- ++(-30:0.8);
    \draw[edge] (bpp) -- ++(-90:0.8);
}}

\newcommand{\graphabpsil}{%
\tikz[baseline=-0.6ex]{
    \node[doublevertex] (a) at (0,0) {$a$};
    \node[smallvertex] (b) at (1.2,0) {$b$};

    \draw[edge] (a) -- (b);

    \draw[->,thick] ($(a)!0.28!(b)$) -- ($(a)!0.50!(b)$);

    \draw[edge] (a) -- ++(90:0.8);
    \draw[edge] (a) -- ++(150:0.8);
    \draw[edge] (a) -- ++(-150:0.8);

    \draw[edge] (b) -- ++(30:0.8);
    \draw[edge] (b) -- ++(-30:0.8);
    \draw[edge] (b) -- ++(-90:0.8);
}}

\newcommand{\graphdeltairrone}{%
\tikz[baseline=-0.6ex]{
    \node[circle, draw, double, minimum size=16pt, inner sep=0pt] (c) at (0,0) {$1$};
    \node[font=\small] at ($(c)+(0.30,0.50)$) {$\delta_{irr}$};

    \foreach \angle in {90,30,-30,-90,-150,150} {
        \draw[thick] (c) -- ++(\angle:0.8);
    }
}
}

\newcommand{\graphsplitonezerodeltairr}{%
\tikz[baseline=-0.6ex]{
    \node[smallvertex] (a) at (0,0) {$0$};
    \node[doublevertex] (b) at (1.2,0) {$1$};

    \draw[edge] (a) -- (b);

    \node[font=\small] at ($(b)+(0.10,0.55)$) {$\delta_{irr}$};

    \draw[edge] (a) -- ++(90:0.8);
    \draw[edge] (a) -- ++(150:0.8);
    \draw[edge] (a) -- ++(-150:0.8);

    \draw[edge] (b) -- ++(30:0.8);
    \draw[edge] (b) -- ++(-30:0.8);
    \draw[edge] (b) -- ++(-90:0.8);
}}

\newcommand{\graphabpsir}{%
\tikz[baseline=-0.6ex]{
    \node[smallvertex] (a) at (0,0) {$a$};
    \node[doublevertex] (b) at (1.2,0) {$b$};

    \draw[edge] (a) -- (b);

   \draw[->,thick] ($(a)!0.72!(b)$) -- ($(a)!0.50!(b)$);

    \draw[edge] (a) -- ++(90:0.8);
    \draw[edge] (a) -- ++(150:0.8);
    \draw[edge] (a) -- ++(-150:0.8);

    \draw[edge] (b) -- ++(30:0.8);
    \draw[edge] (b) -- ++(-30:0.8);
    \draw[edge] (b) -- ++(-90:0.8);
}}

\newcommand{\Mb}{\overline{\mathcal M}}
\newcommand{\Mbargn}{\overline{\mathcal M}_{g,n}}

\newtheorem{theorem}{Theorem}[section]
\newtheorem{proposition}[theorem]{Proposition}

\newtheorem{lemma}[theorem]{Lemma}
\newtheorem{corollary}[theorem]{Corollary}

\theoremstyle{definition}
\newtheorem{definition}[theorem]{Definition}
\newtheorem{example}[theorem]{Example}

\theoremstyle{remark}

\newtheorem{remark}[theorem]{Remark}

\author[A.~Bradford]{Amy Bradford}

\title[Low-weight cohomology of $\Mct$ and the Torelli morphism]{Low-weight compactly supported cohomology of $\Mct$ and the extended Torelli morphism}

\begin{document}

\begin{abstract}
We study the weight zero and weight two graded pieces of the compactly
supported cohomology of the moduli space
$\mathcal M_{g,n}^{ct}$ of stable $n$-pointed curves of genus $g$ of compact type and
show that they are computed by bridgeless graph complexes. In weight two, we construct a natural embedding of shifted weight zero
cohomology from lower genus. For unmarked curves, we construct canonical
classes represented by pinched odd wheel graphs and prove that distinct
wheel monomials give linearly independent classes in the dual of
$W_0H_c^{2g}(\mathcal M_g^{ct};\mathbb Q)$. By relating the
bridgeless graph complex to the $K$-theory of graphs and the Quillen
spectral sequence, we show that these classes remain linearly independent
after pushforward along the extended Torelli morphism
$\mathcal M_g^{ct}\to\mathcal A_g$. Finally, we prove that the
realizations of the motivic Lie algebra arising from $\mathcal M_g$
and $\mathcal A_g$ have the same image in the weight zero compactly supported cohomology
of $\mathcal M_g^{ct}$.
\end{abstract}

\maketitle

\section{Introduction}

Let $g,n\geq 0$ with $2g-2+n>0$, and let
$\mathcal M_{g,n}^{ct}$ denote the moduli space of stable
$n$-pointed curves of genus $g$ of compact type. A stable curve $C$ is of \emph{compact type} if the Picard variety of line bundles of fixed multidegree on $C$ is compact. Equivalently, $C$ is of compact type if its dual graph is a tree. Thus $\mathcal M_{g,n}^{ct}$ is a natural partial
compactification of $\mathcal M_{g,n}$
$$
\mathcal M_{g,n}
\subseteq
\mathcal M_{g,n}^{ct}
\subseteq
\overline{\mathcal M}_{g,n}.
$$

Deligne's mixed Hodge theory equips $H_c^k(\Mct;\QQ)$ with an increasing \textit{weight filtration};
see \cite[$\mathsection$3.2]{DeligneHodgeII} and
\cite[$\mathsection$8.2--8.3]{DeligneHodgeIII}. 
We write $$
\mathrm{gr}_j^WH_c^k(\Mct; \QQ)  := W_jH_c^k(\Mct; \QQ)  / W_{j-1}H_c^k(\Mct; \QQ) 
$$ 
for its \emph{weight $j$ graded piece}.

In this paper, we study the weight zero and the weight two graded pieces of this filtration. The key combinatorial feature is that the boundary complex of $\Mct \subset \Mbargn$ is indexed by bridgeless stable graphs; see Section \ref{subsec: dual complex of boundary}. Building on the weight zero graph complex model for $\M_{g,n}$
\cite{ChanGalatiusPayne2021,ChanGalatiusPayne2022} and the weight two analysis of \cite{PayneWillwacherWeight2}, we construct compact type analogs computing
$$
W_0H_c^\bullet(\mathcal M_{g,n}^{ct};\QQ)
\qquad\text{and}\qquad
\operatorname{gr}_2^W
H_c^\bullet(\mathcal M_{g,n}^{ct};\QQ).
$$
The weight zero model produces a polynomial family of classes represented by pinched odd wheel graphs, while the weight two model contains a naturally embedded copy of shifted weight zero cohomology in lower genus.

For $n=0$, the proper extended Torelli morphism
$$
\tau_g^{ct}: \M_{g}^{ct} \to \A_g
$$
makes $\M_g^{ct}$ a geometric bridge between $\M_g$ and the moduli space $\A_g$ of principally polarized abelian varieties of dimension $g$. We use this bridge to follow the odd wheel classes from $\M_g$ through $\M_g^{ct}$ to $\A_g$, and ultimately to compare the two cohomological realizations of the motivic Lie algebra. 

\subsection{Graph complexes in low-weight}

Our first main result makes precise the appearance of bridgeless graphs in the compact type setting.  Recall that a \textit{bridge} is an edge of a graph whose deletion increases the number of connected components. We apply Deligne's weight spectral sequence
\cite[$\mathsection$3.2]{DeligneHodgeII} to the inclusion
$\mathcal M_{g,n}^{ct}\subseteq
\overline{\mathcal M}_{g,n}$ and identify $\operatorname{gr}_k^WH_c^\bullet
(\mathcal M_{g,n}^{ct};\QQ)$
with the cohomology of a certain graph complex denoted
$\mathrm{GK}_{g,n}^{{bl},k}$; see Subsection \ref{subsec: GK graph complex}. Its generators are bridgeless stable graphs whose vertices
are decorated by classes in the cohomology of the corresponding
moduli spaces $\Mb_{h,m}$, with the vertex decorations having
total cohomological degree $k$. In Subsection \ref{subsec: Modular cooperads}, we
formulate this construction using the language of modular cooperads; see \cite{GetzlerKapranov1998} and
\cite[$\mathsection$2.5]{PayneWillwacherWeight2}.

In weight zero, we show the resulting model reduces to the dual of the marked analog of the simple bridgeless commutative graph complex studied in \cite{WillwacherBridgeless}. We denote this complex by $\GKbarzero$. This gives the compact type analog of the weight zero result in \cite{ChanGalatiusPayne2021}.

\begin{theorem}\label{thm:intro-weight zero-model}
For every stable pair $(g,n)$, there is a natural isomorphism
$$
W_0H_c^\bullet
\bigl(\mathcal M_{g,n}^{ct};\QQ\bigr)
\cong
H^\bullet
\bigl(\overline{\mathrm{GK}}_{g,n}^{{bl},0}\bigr).
$$

\end{theorem}

We next turn to weight two. Since
$H^1(\overline{\mathcal M}_{h,m};\QQ)=0$
for every stable pair $(h,m)$, the weight one graded piece vanishes. 
Thus the weight two graded piece is the first potentially nonzero contribution beyond weight zero. 

To analyze it, we adapt the methods of \cite{PayneWillwacherWeight2} to the bridgeless setting, using the descriptions
and pullback formulas for divisor classes on
$\overline{\mathcal M}_{h,m}$ from
\cite{ArbarelloCornalba1998,Keel1992}. We find that the weight two bridgeless graph complex has two contributions. 
One arises from genus one vertices decorated by the irreducible boundary divisor $\delta_{irr}$, and the other is encoded by a reduced
genus zero graph complex, denoted
$\overline{\mathrm{GK}}_{g,n}^{{bl},2}$.
The precise construction is given in Section~\ref{sec: weight two compactly supported cohomology of Mct}.

Our second main theorem separates these contributions and identifies the genus one term with weight zero cohomology.

\begin{theorem}\label{thm:intro-weight two-model}
Let $g\geq1$ and $2g+n>3$. For every $k$, there is an
isomorphism 
$$
\operatorname{gr}_2^W
H_c^k
\bigl(\mathcal M_{g,n}^{ct};\QQ\bigr)
\cong
W_0H_c^{k-2}
\bigl(\mathcal M_{g-1,n+1}^{ct};\QQ\bigr)
\oplus
H^k
\bigl(\overline{\mathrm{GK}}_{g,n}^{{bl},2}\bigr).
$$
\end{theorem}
Although the direct sum splitting is noncanonical, our construction isolates the weight zero term naturally. In particular, we obtain an injection
$$
W_0H_c^k
\bigl(\mathcal M_{g-1,n+1}^{ct};\QQ\bigr)
\hookrightarrow
\operatorname{gr}_2^W
H_c^{k+2}
\bigl(\mathcal M_{g,n}^{ct};\QQ\bigr).
$$

Thus every nonvanishing result or dimension bound in weight zero produces a
corresponding result in weight two.

\begin{remark}
For comparison, \cite[Theorem~1.2]{PayneWillwacherWeight2} expresses
$\operatorname{gr}_2^W H_c^\bullet(\mathcal M_g;\QQ)$ as a direct sum of
four terms constructed from weight zero cohomology. In the compact type setting, the
reduced complex
$\overline{\mathrm{GK}}_{g,n}^{{bl},2}$
remains as an additional contribution that we do not identify entirely in
terms of weight zero. Nevertheless, the naturally embedded weight zero term isolated above is
already sufficient to transfer nonvanishing results, dimension bounds, and
distinguished graph classes from weight zero to weight two.
\end{remark}

When $n=0$, we combine Theorem~\ref{thm:intro-weight two-model} with the injectivity of pullback along the
universal curve 
$
\mathcal M_{g,1}^{ct}
\longrightarrow
\mathcal M_g^{ct}
$;
see \cite{PayneWillwacherWeight2}. This gives the following corollary.
\begin{corollary}
    For $g \geq 2$, there is an injection
    $
    W_0H_c^k
\bigl(\mathcal M_g^{ct};\QQ\bigr)
\hookrightarrow
\operatorname{gr}_2^W
H_c^{k+2}
\bigl(\mathcal M_{g+1}^{ct};\QQ\bigr).
    $
\end{corollary}

To put this injection to use, we next study weight zero cohomology by
comparing $\mathcal M_g^{{ct}}$ with $\mathcal M_g$ and
$\mathcal A_g$.

\subsection{Comparison with $\mathcal M_g$}

The weight zero compactly supported cohomology of $\M_g$ is governed by the commutative graph complex. More precisely,
\cite{ChanGalatiusPayne2021} gives natural isomorphisms
\begin{equation}\label{eq: Classical weight zero is GC2}
W_0H_c^{2g+k}(\mathcal M_g;\QQ)
\cong
H^k(\mathsf{GC}_2)^{(g)}.
\end{equation}

Graph homology therefore produces large families of unstable
cohomology classes on $\M_g$. The odd wheel classes occur at the
bottom of the possible range. Indeed,
$$
H_c^i(\M_g; \QQ)=0
\quad\text{for }i<2g,
\qquad\text{and}\qquad
H^k(\mathsf{GC}_2)=0
\quad\text{for }k<0.
$$
The first vanishing follows from Harer's virtual cohomological dimension
calculation together with vanishing in the virtual cohomological
dimension
\cite{Harer1986,ChurchFarbPutman2012,MoritaSakasaiSuzuki2013}.
The second is a theorem of Willwacher
\cite[Theorem~1.1]{Willwacher2015} and, through
\eqref{eq: Classical weight zero is GC2}, gives an independent proof of
the corresponding weight zero vanishing.

For every odd $g\geq3$, the wheel graph $W_g$ represents a nonzero
class in the homology of the dual graph complex
\cite[$\mathsection$9]{Willwacher2015}; see also
\cite[Theorem~2.6]{ChanGalatiusPayne2021}. Our next result identifies the smallest degree of nonvanishing
for the compactly supported cohomology of compact type.

\begin{theorem}\label{Thm: intro M_g comparison}
    For every $g \geq 2$, the smallest degree $k$ for which $H_c^k(\M_g^{ct};\QQ)$ is nonzero is $k=2g$.
\end{theorem}

Let $i: \M_g \hookrightarrow \M_g^{ct}$ denote the open inclusion. The vanishing in the theorem follows directly from showing the natural map
$$
i_!\colon
H_c^k(\mathcal M_g;\QQ)
\longrightarrow
H_c^k(\mathcal M_g^{ct};\QQ)
$$
is an isomorphism for $k<2g$ and is injective for $k=2g$. Indeed, 
$
H^k_c(\M_g;\QQ)=0
$
for $k <2g$, and hence the same is true for compact type. Nonvanishing in degree $2g$ follows by Poincaré duality from the
nonzero socle of the tautological ring in cohomological degree $4g-6$
\cite[Proposition~3]{FaberPandharipandeRelativeMaps};
see also
\cite[Theorem~$\star$ and $\mathsection$5.6]
{GraberVakilRelativeLocalization}.

For every odd $g\geq3$, the injectivity of $i_!$ in degree
$2g$ also shows that cohomology classes detected by the odd
wheel graph remain nonzero on $\mathcal M_g^{ct}$. We strengthen this in
Theorem~\ref{thm:intro-wheel-polynomial} by showing that distinct odd
wheel monomials define linearly independent classes in the homology of
the dual bridgeless graph complex. In particular, these classes give
another proof of nonvanishing in degree $2g$ for $g=3$ and every
$g\geq5$. To prove this stronger result, we compare
$\mathcal M_g^{ct}$ with $\mathcal A_g$ through the extended
Torelli map.

\subsection{The extended Torelli map} 

The Torelli map
$$
\tau_g: \M_g \to \A_g
$$
sends a smooth curve to its Jacobian. Since $\tau_g$ is not proper, it does not give rise to a natural morphism of mixed Hodge structures between $H_c^*(\M_g)$ and $H_c^*(\A_g)$. However, the Torelli map factors as 
$$
\M_g \xrightarrow[]{i} \M_g^{ct} \xrightarrow[]{\tau_g^{ct}} \A_g.
$$
Here $i: \M_g \to \M_g^{ct}$ is the open inclusion and $\tau_g^{ct}$ is a proper extension of $\tau_g$. This gives rise to a zigzag of mixed Hodge structures 
\begin{equation}\label{Comparison Map Mg to Ag}
    H_c^*(\M_g; \QQ) \xrightarrow[]{i_!} H_c^*(\M_g^{ct};\QQ) \xleftarrow[]{(\tau_g^{ct})^*} H_c^*(\A_g;\QQ).
\end{equation}

The weight zero compactly supported cohomology of $\M_g$ and
$\A_g$ has been studied extensively; see
\cite{ChanGalatiusPayne2021,ChanGalatiusPayne2022,
BrandtEtAlTopWeightAg}. Several structural results for $\A_g$ parallel known results
for $\M_g$, although neither collection of results appears to follow
directly from the other. The zigzag \eqref{Comparison Map Mg to Ag} suggests a geometric comparison, but the limited understanding of the middle term $\M_g^{ct}$ had previously prevented this comparison from being made effective; see 
\cite[$\mathsection$1.5]{BrownChanGalatiusPayneHopf}. In Section \ref{sec: Comparing Mg and Ag through compact type}, we use the bridgeless graph model to make this comparison explicit on the subspaces generated by odd wheel
classes. This both reveals new structure in $W_0H_c^*(\M_g^{ct})$ and allows us to trace the wheel classes between the curve and abelian sides.

\subsubsection{Polynomial family of wheel classes}

For an odd integer $h\geq3$, let $W_h$ denote the $h$-wheel. Let
$\lambda=(h_1,\ldots,h_r)$ be a partition of $g$ into odd integers
$h_i\geq3$. A \emph{pinched wheel graph of type $\lambda$} is any graph
obtained from
$$
W_{h_1}\sqcup\cdots\sqcup W_{h_r}
$$
by repeatedly identifying a vertex of one connected component with
a vertex of another until the resulting graph is connected. Equivalently, these identifications can be encoded by a tree on the wheel factors together with a choice of attachment vertex for each incidence; see Definition \ref{def:wheel-gluing-datum}. In Section \ref{subsec: Polynomial classes generated by odd wheels}, we prove every such graph is a cycle in
the dual bridgeless graph complex and that any two pinched wheel graphs of the same type are homologous. Consequently, each wheel monomial $W_{h_1}\cdots W_{h_r}$ determines a canonical class
$$
[W_\lambda]
\in
\left(
W_0H_c^{2g}
(\mathcal M_g^{ct};\QQ)
\right)^\vee;
$$
see Proposition \ref{prop:pinched-wheel-independence}. 
Let
$$
\mathcal W:=\QQ[W_3,W_5,W_7,\ldots]
$$
be the genus-graded polynomial algebra generated by the odd wheel graphs,
with $W_h$ in genus $h$, and let $\mathcal W_g$ denote its
genus $g$ summand. Independence of the pinching data produces a well-defined linear map
$$
\Phi_g:\W_g\to (W_0H_c^{2g}(\M_g^{ct};\QQ))^\vee, \qquad W_{h_1}\cdots W_{h_r} \mapsto [W_\lambda].
$$
Our next theorem shows these classes satisfy no further linear relations.

\begin{theorem}\label{thm:intro-wheel-polynomial} Together with the assignment $\Phi(1)=1$, the maps $\Phi_g$
assemble into an injective genus graded linear map
$$
\Phi\colon
\QQ[W_3,W_5,W_7,\ldots]
\hookrightarrow \QQ\oplus
\bigoplus_{g\geq 2}\left(
W_0H_c^{2g}
(\mathcal M_g^{ct};\QQ)
\right)^\vee.
$$
\end{theorem}

The heart of the proof draws on the structure of the weight zero cohomology of
$\mathcal A=\bigsqcup_{g\geq0}\mathcal A_g$. In \cite{BrownChanGalatiusPayneHopf}, it is shown
 that $W_0H_c^\bullet(\mathcal A;\QQ)$ carries a bigraded Hopf algebra
structure whose graded dual contains a polynomial algebra
generated by the odd wheels. Their construction compares the Quillen spectral
sequence for $K(\ZZ)$ with a compatible graph spectral sequence.

We place the dual bridgeless graph complex in this framework through the
map found in Lemma~\ref{lem:bridgeless-inclusion-into-GrC}. By doing so, we
compare the pinched wheel classes with the corresponding wheel monomials
on the $\mathcal A_g$ side. The linear independence of these monomials,
proved in \cite{BrownChanGalatiusPayneHopf}, implies
Theorem~\ref{thm:intro-wheel-polynomial}. We later show in Lemma \ref{lem:bridgeless-graph-complex-to-Quillen} that this
comparison is compatible with pushforward under the extended Torelli
morphism. We prove further in Theorem \ref{thm: wheel classes inject under torelli} that the extended Torelli pushforward is injective on the wheel subspace $\mathrm{im}(\Phi)$.

\subsubsection{Comparison of two motivic families}

The odd wheel classes also detect a broader Lie-theoretic structure. Theorems 1.2 and 1.3 of \cite{BrownChanGalatiusPayneHopf} construct canonical
primitive classes
$$
\omega^{4k+1}
\in
W_0H_c^{4k+2}(\mathcal A_{2k+1};\mathbb R),
\qquad k\geq1,
$$
and show the first eleven classes
$\omega^5,\omega^9,\ldots,\omega^{45}$ generate a free Lie
subalgebra. 

In \cite[Proposition~1.8]{BrownChanGalatiusPayneHopf}, the authors
construct a morphism of graded Lie algebras from the diagonal primitive
weight zero cohomology of $\mathcal A$ to
$H^0(\mathsf{GC}_2)\otimes\mathbb R$. Under this morphism, the classes
$\omega^{4k+1}$ map to canonical classes in graph cohomology; see also
\cite[$\mathsection$5.2]{BrownPanzerPortner2026}. The isomorphism
$$
H^0(\mathsf{GC}_2)\cong\mathfrak{grt}_1,
$$
established in \cite[Theorem~1.1]{Willwacher2015}, identifies these classes with elements of the Grothendieck--Teichm\"uller Lie algebra.
Their nonzero pairings with the odd wheel graphs
\cite{BrownSchnetzWheelClasses} together with \cite[Theorem~1.2]{WillwacherFreeGenerators} show that the full family $\omega^5,\omega^9,\omega^{13},\ldots$ generates a free Lie subalgebra.

Let $\mathfrak g^{\mathfrak m}$ denote the graded Lie algebra associated
to the category of mixed Tate motives over $\mathbb Z$. It is freely generated by one element in each odd degree $2k+1 \geq 3$. Combining the embedding
$\mathfrak g^{\mathfrak m}\hookrightarrow\mathfrak{grt}_1$
of \cite{Brown2012} with the graph complex description \eqref{eq: Classical weight zero is GC2} realizes $\mathfrak{g^m}\otimes_\QQ \RR$ in the diagonal weight zero compactly supported cohomology of $\M_g$. Let
$$
\L_\M\subset \bigoplus_{g\geq 2}W_0H_c^{2g}(\M_g;\RR)
$$
denote the resulting realization. On the abelian
side, set
$$
\L_{\A}:=\mathbb{L}(\omega^5,\omega^9,\omega^{13},\ldots) \subseteq \bigoplus_{g\geq 2} W_0H_c^{2g}(\A_g;\RR).
$$
By \cite[Theorem~3]{BrownPanzerPortner2026}, under
Willwacher's identification with $\mathfrak{grt}_1$,
the canonical graph classes arising from the
$\omega^{4k+1}$ lie in and generate the motivic Lie
subalgebra
$\mathfrak g^{\mathfrak m}\otimes_{\QQ}\RR$. Thus $\L_\A$ and $\L_\M$ are two cohomological realizations of the same motivic Lie algebra, obtained through a priori different constructions. 

These constructions do not by themselves show that the two realizations are compatible with the algebraic Torelli map; this question is noted in
\cite[$\mathsection$1.5]{BrownChanGalatiusPayneHopf} and
\cite[Introduction, following Theorem~3]{BrownPanzerPortner2026}.
Our final theorem shows that these two families agree after mapping to the compactly supported
cohomology of the compact type locus.

\begin{theorem}\label{thm:intro-motivic-comparison}
Under the zigzag
$$
\bigoplus_{g\geq2}
W_0H_c^{2g}(\mathcal A_g;\mathbb R)
\xrightarrow{\,(\tau^{ct})^*\,}
\bigoplus_{g\geq2}
W_0H_c^{2g}(\mathcal M_g^{ct};\mathbb R)
\xleftarrow{\,i_!\,}
\bigoplus_{g\geq2}
W_0H_c^{2g}(\mathcal M_g;\mathbb R),
$$
the two motivic families have the same image
$
(\tau^{ct})^*(\L_{\mathcal A})
=
i_!(\L_{\mathcal M}).
$
Furthermore, $(\tau^{ct})^*|_{\L_{\mathcal A}}$ is injective.

\end{theorem}

\subsection{Acknowledgments} The author would like to thank Thomas Willwacher, Melody Chan, Samir Canning, Siddarth Kannan, Terry Song, and Thomas Wennink for enlightening conversations related to this project. The author is especially grateful to their doctoral advisor, Sam Payne, for his generous mentorship and for many insightful conversations throughout this project. The author was supported by the National Science Foundation Graduate
Research Fellowship Program under Grant No.~DGE-2137420.

\section{Background and Setup}

\subsection{Conventions} Unless specifically stated otherwise, all vector spaces are understood to be $\QQ$ vector spaces and all homology and cohomology groups will be taken with $\QQ$-coefficients. If $V$ is a finite dimensional vector space with an action of a finite group $G$, we will identify the invariant subspace $V^G$ with the coinvariant space $V_G$.

When $V$ is additionally a graded vector space, for any $k \in \ZZ$, define $V[k]$ to be the graded vector space whose degree $j$ part is the degree $j+k$ part of $V$. Here, we use the standard Koszul sign convention. Specifically, if $V$ and $W$ are two graded vector spaces, the isomorphism
$$
V \otimes W \ \xrightarrow{\sim} W \otimes V
$$
sends an element $v \otimes w$, consisting of homogeneous elements $v \in V$, $w\in W$ of degrees $|v|$ and $|w|$, respectively, to the element $(-1)^{|v||w|}w \otimes v$. If $f\in \mathrm{Hom}(V_1, V_2)$ and $g \in \mathrm{Hom}(W_1,W_2)$ are two linear maps of graded vector spaces, of homogeneous degree $|f|$ and $|g|$ respectively, then define the linear map $f \otimes g \in \mathrm{Hom}(V_1 \otimes W_1, V_2 \otimes W_2)$ to be
$$
(f \otimes g)(v \otimes w)= (-1)^{|g||v|} f(v) \otimes g(w)
$$
for all homogeneous elements $v \in V_1$, $w\in W_1$.

If $V$ is additionally a differential graded vector space, we abbreviate the term ``differential graded'' to dg. For the bulk of the paper, we consider cohomological grading conventions so that all differentials have degree $+1$ unless specifically stated otherwise. When $(V, d_V)$ and $(W,d_W)$ are dg vector spaces, we denote the differential on $V \otimes W$ as
\begin{equation*}
    d_V + d_W:= d_V \otimes \mathrm{id}_W + \mathrm{id}_V \otimes d_W,
\end{equation*}
 remembering the Koszul sign conventions. 

We often work with moduli spaces and graph complexes indexed by genus $g$ and some number of marked points $n$. If $Y_{g,n}$ is any such space, we denote $Y_g:= Y_{g,0}$.

\subsection{Deligne's weight spectral sequence.}

We recall the Poincaré dual form of Deligne's weight spectral sequence; see
\cite[$\mathsection$3.2]{DeligneHodgeII} and, for this formulation,
\cite[Example~3.5]{PetersenSpectralSequence}
and \cite[$\mathsection$2.3]{PayneWillwacherWeight2}.
Let $X$ be a smooth complex algebraic variety or Deligne--Mumford stack
with a smooth proper compactification
$
X\subseteq \overline X
$
whose boundary $D=\overline X\setminus X$ is a normal crossings divisor.
For $p\geq 1$, let $\widetilde D^{(p)}$ denote the disjoint union of
the normalizations of the closures of the codimension-$p$ boundary
strata, with the unordered set of $p$ local branches defining each
stratum retained through normalization. Set
$\widetilde D^{(0)}=\overline X$.
 The permutation action on the $p$
local branches along each stratum determines a rank-one sign local system
$\varepsilon^{(p)}$, with $\varepsilon^{(0)}=\QQ$.

There is a spectral sequence
$$
E_1^{p,k}
=
H^k\!\left(\widetilde D^{(p)};\varepsilon^{(p)}\right)
\Longrightarrow
H_c^{p+k}(X),
$$
whose first differential is the alternating sum of the pullbacks obtained
by forgetting one boundary branch. Since the terms in the $k$-th row are
pure of weight $k$, the spectral sequence degenerates at $E_2$, giving
canonical identifications
$$
\operatorname{gr}_k^W H_c^{p+k}(X)
\cong
E_2^{p,k}
\cong
H^p\!\left(E_1^{\bullet,k},d_1\right).
$$

We now specialize this construction to the pair
$
\Mct\subset\Mbar.
$
Set $D:=\Mbar\setminus\Mct$. For $g \geq 1$, the divisor $D$ is the irreducible divisor $\Delta_{\mathrm{irr}}$.  In Proposition~\ref{Boundary indexed by bridgeless graphs}, we show the components of $\widetilde{D}^{(p)}$ are indexed by bridgeless stable $n$-marked graphs of
genus $g$ with $p$ edges.
For such a graph $\Gamma$, set
$$
\M_\Gamma
:=
\prod_{v\in V(\Gamma)}\M_{g_v,n_v},
\qquad
\overline{\M}_\Gamma
:=
\prod_{v\in V(\Gamma)}\overline{\M}_{g_v,n_v},
$$
where $g_v=w(v)$ is the genus of the vertex $v$ and $n_v$ is the number of half-edges incident to $v$.

The natural gluing morphism
$
\xi_\Gamma\colon\M_\Gamma\longrightarrow\Mbar
$
extends to $\overline{\M}_\Gamma$. Its image
$
D_\Gamma
:=
\xi_\Gamma\left(\overline{\M}_\Gamma\right)
$
is the closure of the locus of curves with dual graph $\Gamma$, and
its normalization is the smooth proper Deligne--Mumford stack
$$
\widetilde D_\Gamma
\cong
\left[
\overline{\M}_\Gamma/\operatorname{Aut}(\Gamma)
\right].
$$
Consequently,
$$
\widetilde D^{(p)}
\cong
\bigsqcup_{\substack{[\Gamma]\ \mathrm{bridgeless}\\
|E(\Gamma)|=p}}
\widetilde D_\Gamma,
$$
where the union ranges over isomorphism classes of stable
$n$-marked bridgeless graphs of genus $g$.

After pullback to $\overline{\mathcal M}_\Gamma$, the $p$ retained
boundary branches are globally indexed by the edges of $\Gamma$.
Consequently, 
the local system $\varepsilon^{(p)}$ is identified with the local system
associated to the determinant of the permutation representation of
$\operatorname{Aut}(\Gamma)$ on $E(\Gamma)$; see
\cite[$\mathsection$2.3]{PayneWillwacherWeight2}. 

Thus the $k$-th row of the $E_1$-page of the weight spectral
sequence is naturally identified with the complex
$C_k^\bullet$, whose $j$-th term is
$$
C_k^j:=
\bigoplus_{\substack{[\Gamma]\ \mathrm{bridgeless}\\ |E(\Gamma)|=j}}
\left(
H^k(\overline{\M}_\Gamma)
\otimes
\det(E(\Gamma))
\right)^{\operatorname{Aut}(\Gamma)}.
$$
Consequently, there is a canonical isomorphism of $\operatorname{gr}_k^W H_c^{j+k}(\Mct)$ with $H^j$ of the following complex
\begin{equation}\label{eq: gr_k DWSS complex}
\cdots
\longrightarrow
C_k^{j-1}
\xrightarrow{\,d_{j-1}\,}
C_k^j
\xrightarrow{\,d_j\,}
C_k^{j+1}
\longrightarrow
\cdots.
\end{equation}

The language of modular cooperads and Feynman transforms packages the rows
of this spectral sequence as decorated graph complexes. We recall the necessary
formalism next.

\subsection{Modular cooperads and the Feynman transform}\label{subsec: Modular cooperads}
We briefly recall the definitions of modular cooperads and their
Feynman transforms; see
\cite[$\mathsection$2.4]{PayneWillwacherWeight2} and
\cite{GetzlerKapranov1998}. A \textit{stable $\mathbb S$-module} is
a collection of dg vector spaces
$$
\P=\{\P(g,n)\}_{2g-2+n>0},
$$
where each $\P(g,n)$ is equipped with an action of the symmetric
group $S_n$. For any stable graph $\Gamma$, write
$$
\P(\Gamma)
:=
\bigotimes_{v\in V(\Gamma)}
\P(g(v),n(v)).
$$
The automorphism group
$\operatorname{Aut}(\Gamma)$ acts naturally on $\P(\Gamma)$ by
permuting the vertices and the half-edges.

A \textit{modular cooperad} is a stable $\mathbb S$-module equipped,
for every stable $n$-marked graph $\Gamma$ of genus $g$, with a map
$$
\P(g,n)\longrightarrow\P(\Gamma),
$$
satisfying certain compatibility conditions.

Let $\P$ be a dg modular cooperad. Its \textit{Feynman transform},
denoted $\F\P$, is a modular cooperad whose underlying stable
$\mathbb S$-module is
$$
\F\P(g,n)
:=
\bigoplus_{[\Gamma]}
\left(
\P(\Gamma)
\otimes
\QQ[-1]^{\otimes |E(\Gamma)|}
\right)_{\operatorname{Aut}(\Gamma)}.
$$
Here the sum ranges over isomorphism classes of stable $n$-marked
graphs of genus $g$, and one copy of $\QQ[-1]$ is assigned to each
edge. The automorphism group acts diagonally on $\P(\Gamma)$ and on
the edge factors, permuting the latter with the Koszul signs.

Elements of $\F\P(g,n)$ may be represented by stable
graphs $\Gamma$, together with a decoration
$
\alpha_v\in\P(g(v),n(v))
$
at each vertex and a total ordering of $E(\Gamma)$. Reordering the
edges multiplies the corresponding generator by the sign of the
permutation. When $\P$ has no differential, the differential on $\F\P$ is
given by two pieces:
$$
\delta=\deltasp+\deltal.
$$
The term $\deltasp$ splits a vertex into two vertices joined by a new
edge, distributing the genus and half-edges between them so that
both new vertices are stable. The corresponding cooperad structure
map is applied to the decoration of the original vertex

$$
\begin{tikzpicture}[
  baseline=(current bounding box.center),
  vertex/.style={
    circle,
    draw,
    fill=white,
    inner sep=0pt,
    minimum size=4pt
  },
  edge/.style={line width=0.5pt}
]
  \node at (-1.35,0) {$\deltasp\colon$};

  \coordinate (v) at (0,0);
  \draw[edge] (v) -- (-0.65, 0.65);
  \draw[edge] (v) -- (-0.65,-0.65);
  \draw[edge] (v) -- ( 0.65, 0.65);
  \draw[edge] (v) -- ( 0.65,-0.65);
  \draw[edge] (v) -- ( 0.65, 0);
  \draw[edge] (v) -- ( -0.65,0);
  \node[vertex] at (v) {};

  \node at (1.30,0) {$\longmapsto$};
\node at (2.4,0)
  {$\displaystyle\sum$};
\coordinate (v1) at (3.65,0);
\coordinate (v2) at (4.45,0);
\draw[edge] (v1) -- (v2);
\draw[edge] (v1) -- (3.00, 0.65);
\draw[edge] (v1) -- (3.00,-0.65);
\draw[edge] (v1) -- (3.00,0);
\draw[edge] (v2) -- (5.10, 0.65);
\draw[edge] (v2) -- (5.10,-0.65);
\draw[edge] (v2) -- (5.10,0);
\node[vertex] at (v1) {};
\node[vertex] at (v2) {};
\end{tikzpicture}.
$$
The term $\deltal$ creates a loop at a vertex

$$
\begin{tikzpicture}[
  baseline=(current bounding box.center),
  vertex/.style={
    circle,
    draw,
    fill=white,
    inner sep=0pt,
    minimum size=4pt
  },
  edge/.style={line width=0.5pt}
]
  \node at (-1.25,0) {$\deltal\colon$};

  \coordinate (v) at (0,0);
  \draw[edge] (v) -- (-0.65,-0.70);
  \draw[edge] (v) -- (0,-0.78);
  \draw[edge] (v) -- (0.65,-0.70);
  \node[vertex] at (v) {};

  \node at (1.25,0) {$\longmapsto$};

  \coordinate (w) at (2.55,0);
  \draw[edge] (w) -- (1.90,-0.70);
  \draw[edge] (w) -- (2.55,-0.78);
  \draw[edge] (w) -- (3.20,-0.70);
  \draw[edge] (w)
    .. controls (2.18,0.55) and (2.92,0.55) .. (w);
  \node[vertex] at (w) {};
\end{tikzpicture}
$$
using the modular cooperad structure map
$
\P(g,n)
\longrightarrow
\P(g-1,n+2)
$
on its decoration. In both expansion terms, the new edge is placed first in
the edge ordering. 

\subsection{The Getzler--Kapranov graph complex.}\label{subsec: GK graph complex}
The collection of the cohomology groups $H^\bullet(\Mbar)$ forms a modular cooperad, denoted $H(\Mb)$ \cite[$\mathsection$~2.5]{PayneWillwacherWeight2}.

Elements of the Feynman transform, denoted
$\F H(\Mb)(g,n)$, are linear combinations of stable $n$-marked
graphs of genus $g$, whose vertices $v$ are decorated by classes in
$$
H^\bullet\left(\Mb_{g(v),n(v)}\right).
$$
Since $H(\Mb)$ has zero internal
differential, the differential on its Feynman transform has two parts:
$$
\delta
=
\deltasp+\deltal.
$$

The dg vector spaces $\F H(\Mb)(g,n)$ inherit the weight grading from $H(\Mb)$, where the weight of a graph is the sum of the degrees of its vertex decorations. We consider the weight $k$ part of the Feynman transform whose generators are linear combinations of graphs $\Gamma$ whose vertices $v$ are labeled with elements of $H^{k_v}(\Mb_{g_v,n_v})$ such that $\sum_vk_v=k$. The result is now known as the Getzler--Kapranov graph complex; see, e.g., \cite[$\mathsection$6.1]{AnderssonWillwacherZivkovic2020}, \cite{KaluginGetzlerKapranov2020}, and \cite[$\mathsection$2.5]{PayneWillwacherWeight2}.

\begin{definition}~\cite[Definition~2.2]{PayneWillwacherWeight2}
    The weight $k$ Getzler--Kapranov graph complex is 
    $$
    \mathrm{GK}_{g,n}^k:= \mathrm{gr}^W_k\F H(\Mb)(g,n).
    $$
\end{definition}

The weight $k$ Getzler--Kapranov graph complex naturally computes $\mathrm{gr}_k^WH_c^\bullet(\M_{g,n})$; see \cite{PayneWillwacherWeight2}. To understand the compact type case, consider the dg subspace
$$
\mathrm{GK}^{br,k}_{g,n} \subseteq \mathrm{GK}_{g,n}^k
$$
generated by decorated stable graphs whose underlying graph contains at least one bridge edge. This is a subcomplex because vertex splitting and loop addition
preserve every existing bridge. The quotient by this subcomplex is a graph complex whose underlying generators are decorated stable bridgeless graphs, or a bridgeless version of the Getzler--Kapranov graph complex. 
\begin{definition}
   \textit{The weight $k$ bridgeless Getzler--Kapranov graph complex is }
    \begin{equation*}
        \text{GK}^{bl,k}_{g,n} := \text{GK}^{k}_{g,n}~/~ \text{GK}^{br,k}_{g,n}.
    \end{equation*}
\end{definition}
The weight $k$ bridgeless Getzler--Kapranov graph complex is naturally identified with the weight $k$ row of \eqref{eq: gr_k DWSS complex}. Thus there is a canonical isomorphism
\begin{equation}\label{eq: Canonical isomorphism between the kth graded piece and GKbl,k}
    H^\bullet\GK \cong \text{gr}^W_kH^\bullet_c(\Mct).
\end{equation}
The differential on the quotient can be described as follows. Let $\pi_{bl}: \text{GK}^{k}_{g,n} \to \text{GK}^{bl,k}_{g,n}$ be the quotient map. Then the induced differential on $\text{GK}^{bl,k}_{g,n}$ is
$
\delta^{bl}=\pi_{bl} \circ \delta.
$
That is, one computes the usual Getzler--Kapranov differential and discards every resulting term that contains a bridge. As $\deltal$ cannot add new bridges to a graph, this only applies to the $\deltasp$ piece of the differential. Thus if $\delta^{bl}_{split}:=\pi_{bl}\circ \deltasp$, the differential on the bridgeless complex is
\begin{equation*}\label{Differential}
    \delta^{bl}=\delta^{bl}_{split} + \deltal.
\end{equation*}

When working with the bridgeless complex, we will suppress the notation $\delta^{bl}_{split}$ and write the differential simply as $\delta=\delta_{split}+\deltal$ when no confusion seems possible.

\begin{remark}
In recent work, Bergstr\"om and Wennink give a computer implementation
of Deligne's weight spectral sequences, including the pullback
complexes identified here with the Getzler--Kapranov graph complex
and its bridgeless analogue. See \cite[Sections~3--5]{BergstromWennink2026}
for the construction and implementation.
\end{remark}

\subsection{The weight zero cohomology of $\mathcal A_g$.}
 In \cite{BrownChanGalatiusPayneHopf}, the authors construct a bigraded Hopf algebra structure on 
$$
W_0H_c^*\Big(\coprod_{g \geq 0}\A_g\Big)= \bigoplus_{g \geq 0}W_0H_c^*(\A_g).
$$
The coproduct in this Hopf algebra structure is graded-cocommutative, and is induced by pullback along the proper morphisms
$
\A_g \times \A_h \to 
\A_{g+h}.
$

Let $P_g \subset \mathrm{Sym}^2((\RR^g)^\vee)$ denote the set of symmetric bilinear forms on $\RR^g$ that are positive definite.  Let $P_g^{\mathrm{rt}}$ denote the set of positive semidefinite forms on $\RR^g$ whose kernel is rational, i.e., the kernel is of the form $W \otimes \RR$ for some vector subspace $W \subset \QQ^g$. Topologize $P_g^{\mathrm{rt}}$ with its Satake topology; see \cite[$\mathsection$2.1]{BrownChanGalatiusPayneHopf} for more details. Then define
$
A_g^\mathrm{trop}:= P_g^{\mathrm{rt}}/\mathrm{GL}_g(\ZZ).
$
\begin{proposition}~\cite[Theorem~3.1]{BrandtEtAlTopWeightAg},~\cite[Corollary~2.9]{OdakaOshima2021}
For each $k \geq 0$, there is a canonical isomorphism
$
H_c^k(A_g^\mathrm{trop})\cong W_0H_c^k(\A_g).
$
\end{proposition}
For later use, define

$$
\mathcal A_{g,t}^{\mathrm{BM}}
:=
\left(
W_0H_c^{g+t}(\mathcal A_g)
\right)^\vee
\cong
H_{g+t}^{\mathrm{BM}}
(A_g^{\mathrm{trop}}),
\qquad
\mathcal A^{\mathrm{BM}}
:=
\bigoplus_{g,t}\mathcal A_{g,t}^{\mathrm{BM}}.
$$

Dualizing the coproduct gives a graded-commutative product on
$\mathcal A^{\mathrm{BM}}$. This product
is induced by block sum of quadratic forms
\cite[Proposition~4.1]{BrownChanGalatiusPayneHopf}.

\subsection{The Quillen spectral sequence.}\label{sec:weight zero-Ag}

We recall the homological spectral sequence associated with the rank
filtration on the Waldhausen $S$-construction for the algebraic
$K$-theory of $\ZZ$. Following
\cite{BrownChanGalatiusPayneHopf}, we refer to it as the
\emph{Quillen spectral sequence}. We record only the features needed in
Sections~\ref{subsec: Polynomial classes generated by odd wheels} and
\ref{subsec: Injectivity under the extended Torelli morphism}; see
\cite[Sections~2.2--2.4]{BrownChanGalatiusPayneHopf} and
\cite[II.9 and IV.8]{WeibelKBook2013} for the construction.

Let $BK(\ZZ)$ denote the Waldhausen delooping of the algebraic
$K$-theory space of $\ZZ$. Its rank filtration is denoted by
$$
F_0BK(\ZZ)\subseteq F_1BK(\ZZ)\subseteq\cdots,
$$
where $F_sBK(\ZZ)$ is represented by diagrams whose largest module
has rank at most $s$. The associated homological spectral sequence
has the form
$$
{}^QE^1_{s,t}
=
H_{s+t}\bigl(F_sBK(\ZZ),F_{s-1}BK(\ZZ);\QQ\bigr)
\cong
H_t\bigl(
\operatorname{GL}_s(\ZZ);
\operatorname{St}_s\otimes\QQ
\bigr)
\Longrightarrow
H_{s+t}\bigl(BK(\ZZ);\QQ\bigr),
$$
where $\operatorname{St}_s$ is the Steinberg module; see
\cite[Proposition~2.15]{BrownChanGalatiusPayneHopf}.

The first page also admits a geometric description. Fix the isomorphisms constructed in \cite[Section~2.4 and Remark~2.21]{BrownChanGalatiusPayneHopf}:
\begin{equation}\label{eq: Quillen as PD quadratic forms}
\theta_{s,t}\colon
H_{s+t}^{\mathrm{BM}}
\left(
P_s/\operatorname{GL}_s(\ZZ);\QQ
\right)
\xrightarrow{\sim}
{}^QE^1_{s,t}.
\end{equation}

 Under this comparison, the resulting product on the first
page is induced by block sum of quadratic forms;
see \cite[Proposition~2.15 and Remark~2.22]
{BrownChanGalatiusPayneHopf}. Together with the filtered coproduct, this
makes the Quillen spectral sequence a spectral sequence of bigraded Hopf
algebras. In particular, the product on ${}^QE^1$ is
graded-commutative with respect to total degree
\cite[Theorem~3.18]{BrownChanGalatiusPayneHopf}.

The $E^1$-page is also closely related to the Borel--Moore homology of $A_g^{\mathrm{trop}}$. Let $e\in{}^QE^1_{1,0}$ denote the canonical generator chosen in
\cite[Section~4.2]{BrownChanGalatiusPayneHopf}.

\begin{proposition}[{\cite[Proposition~4.4]
{BrownChanGalatiusPayneHopf}}]
\label{Prop: A^BM injects into quillen}
There is a canonical isomorphism of graded-commutative algebras
$$
\mathcal A^{\mathrm{BM}}
\otimes_{\QQ}
\QQ[\epsilon]/(\epsilon^2)
\xrightarrow{\sim}
{}^QE^1
$$
sending $\epsilon$ to $e$, where $\epsilon$ has bidegree
$(1,0)$.
\end{proposition}
We denote the inclusion of the first tensor factor by
\begin{equation*}\label{eq: i_A: A^BM into QE^1}
\iota_{\mathcal A}\colon
\mathcal A^{\mathrm{BM}}
\lhook\joinrel\longrightarrow
{}^QE^1.
\end{equation*}
For later use, we record this inclusion componentwise. In total degree $k=g+t$, define 
$$\pi_g: H_k^{\mathrm{BM}}(A_g^{\mathrm{trop}})
\longrightarrow
H_k^{\mathrm{BM}}
\left(
A_g^{\mathrm{trop}},
 A_{g-1}^{\mathrm{trop}}
\right)
\xrightarrow{\sim}
H_k^{\mathrm{BM}}
\left(P_g/\operatorname{GL}_g(\mathbb Z)\right),
$$
by the natural map to relative Borel--Moore homology followed by excision; see \cite[Lemma~2.2]{BrownChanGalatiusPayneHopf}.

Under the isomorphism of Proposition~\ref{Prop: A^BM injects into quillen}, the
$(g,t)$-component of $\iota_{\mathcal A}$ is 
\begin{equation}\label{eq:iota-A-concrete}
(\iota_{\mathcal A})_{g,t}=\theta_{g,t}\circ\pi_g.
\end{equation}
This is the form used in the proof of Lemma \ref{lem:bridgeless-graph-complex-to-Quillen}.

\subsection{$K$-theory of graphs.}\label{sec:K-theory-of-graphs}
We recall the graph $K$-theory construction of
\cite[$\mathsection$5]{BrownChanGalatiusPayneHopf}, which compares a
filtered space of graphs with the rank filtration on $BK(\ZZ)$.
We first describe the spaces and the graph Laplacian map, then pass
to cellular chains and the induced comparison of spectral sequences.

The construction gives a space $BK_{\mathrm{Gr}}$ built from
diagrams of subgraphs and edge contractions. Its filtration records
the first Betti number of the largest graph in each diagram.
Applying $H_1(-;\ZZ)$ to these diagrams gives a filtered map
$$
BK_{\mathrm{Gr}}\longrightarrow BK(\ZZ),
$$
since $H_1(G;\ZZ)$ is free of rank $b_1(G)$; see
\cite[Sections~5.1--5.2]{BrownChanGalatiusPayneHopf}.
Passing to isomorphism classes of graph diagrams gives a simpler
space $M_{\mathrm{Gr}}$ and a filtered map
$q:BK_{\mathrm{Gr}}\to M_{\mathrm{Gr}}$.
Thus the spaces fit into a zigzag
$$
M_{\mathrm{Gr}}
\xleftarrow{\ q\ }
BK_{\mathrm{Gr}}
\longrightarrow
BK(\ZZ).
$$
The map $q$ induces isomorphisms on rational homology, including
the homology of the associated graded spaces
\cite[Lemma~5.11]{BrownChanGalatiusPayneHopf}.
We use $M_{\mathrm{Gr}}$ to describe the graph side of this
comparison explicitly.

The space $M_{\mathrm{Gr}}$ has a rational cell decomposition
indexed by finite, possibly disconnected graphs. These graphs carry
no vertex weights, and no stability conditions are imposed.
Bridges, loops, parallel edges, isolated vertices, and the empty
graph are all allowed. The open cell indexed by $G$ is
\begin{equation}\label{eq:MGr-cell}
\sigma_G
=
(0,\infty)^{E(G)}/\operatorname{Aut}(G).
\end{equation}
A point of this cell assigns a positive length to each edge of
$G$, modulo graph automorphisms.
The rational cell decomposition is described by the characteristic maps
$$
f_G:[0,1]^{E(G)}\longrightarrow M_{\mathrm{Gr}},
$$
which factor through the quotient by $\operatorname{Aut}(G)$.
On the cube interior, we may use the length coordinates
$\ell_e=-\log u_e$, where $u_e\in(0,1)$.

The attaching maps are determined by edge contraction and deletion.
On the face $u_e=1$, $f_G$ agrees with the characteristic map
for $G/e$, the graph obtained by \emph{contracting} $e$.
On the face $u_e=0$, it agrees with the characteristic map
for $G\setminus e$, the graph obtained by \emph{deleting} $e$.
Thus edges contract as their lengths tend to zero and are
deleted as their lengths tend to infinity.
See \cite[Corollary~5.13 and Remark~5.14]
{BrownChanGalatiusPayneHopf}.

Neither contracting nor deleting an edge increases the first Betti number of the graph. Hence the cells
with $b_1(G)\leq g$, together with their boundary attachments,
form a filtered subspace
$$
F_gM_{\mathrm{Gr}}
:=
\bigcup_{b_1(G)\leq g}\sigma_G.
$$
We put
$F_{-1}M_{\mathrm{Gr}}=F_{-1}BK_{\mathrm{Gr}}=\varnothing$. For $g=0$, set
$\operatorname{Gr}_0(M_{\mathrm{Gr}}):=(F_0M_{\mathrm{Gr}})_+$,
where the subscript $+$ denotes the addition of a disjoint basepoint.
For $g\geq1$, set
$$
\operatorname{Gr}_g(M_{\mathrm{Gr}})
:=
F_gM_{\mathrm{Gr}}/F_{g-1}M_{\mathrm{Gr}}.
$$
This quotient retains the genus $g$ cells and collapses every
boundary stratum of smaller genus to the basepoint.

We next describe a useful geometric map on $\mathrm{Gr}_g(M_{\mathrm{Gr}})$. Let $G$
satisfy $b_1(G)=g$. After orienting its edges, write
$$
c=\sum_{e\in E(G)}a_e e\in H_1(G;\RR)
$$
and define the coedge $e^\vee\in H^1(G;\RR)$ by
$e^\vee(c)=a_e$. For positive edge lengths
$\ell=(\ell_e)_{e\in E(G)}$, set
$$
Q_{G,\ell}
:=
\sum_{e\in E(G)}
\ell_e(e^\vee)^2.
$$
This is a positive definite quadratic form on $H_1(G; \RR)$, and its definition is independent of the chosen edge orientations.
 An integral basis of
$H_1(G;\ZZ)$ identifies $Q_{G,\ell}$ with an element of $P_g$.
Changes of basis and graph automorphisms act through
$\operatorname{GL}_g(\ZZ)$, giving the graph Laplacian map
$$
L_G:\sigma_G\longrightarrow P_g/\operatorname{GL}_g(\ZZ),
\qquad
\ell\longmapsto[Q_{G,\ell}].
$$
The cell maps are compatible with the attachments that preserve genus.
They extend over the boundary faces of smaller genus by sending
those faces to $\infty$. Consequently, they assemble to a pointed map
\begin{equation}\label{eq:Lg map}
L_g:
\operatorname{Gr}_g(M_{\mathrm{Gr}})
\longrightarrow
\left(P_g/\operatorname{GL}_g(\ZZ)\right)\cup\{\infty\}.
\end{equation}
The source basepoint is sent to $\infty$, and the target is the
one-point compactification of $P_g/\operatorname{GL}_g(\ZZ)$; see
\cite[pp.~70--71]{BrownChanGalatiusPayneHopf}
and \cite{BrownSchnetzWheelClasses}.

The rational cellular chains of $M_{\mathrm{Gr}}$ admit the following graph model.
Let $C=(C_\bullet,\partial)$ be generated over $\QQ$ by symbols
$[G,\omega]$, where $G$ is a finite graph and
$\omega$ is a total ordering of $E(G)$. We impose the relation
$$
[G,\omega]
=
\operatorname{sgn}(\tau)[G',\omega']
$$
whenever an isomorphism $G\to G'$ identifies the edge orderings
by the permutation $\tau$. The degree of $[G,\omega]$ is
$|E(G)|$. Notice that a graph with an automorphism inducing an
odd permutation of its edges represents zero in $C$.

The differential is given by $\partial=\partial_c-\partial_d$. If $\omega=e_1<\cdots<e_r$, then
\begin{equation*}\label{eq:MGr-differential}
\partial_c[G,\omega]
=
\sum_{i=1}^r
(-1)^{i-1}[G/e_i,\omega_i],
\qquad
\partial_d[G,\omega]
=
\sum_{i=1}^r
(-1)^{i-1}[G\setminus e_i,\omega_i].
\end{equation*}
Here $\omega_i$ is the induced ordering of the remaining edges.
Thus the two parts of the differential record the two types of
boundary attachment in $M_{\mathrm{Gr}}$.

For a graph $G$ with $r$ edges and an edge ordering $\omega$, let
$f_{G,\omega}:[0,1]^r\to M_{\mathrm{Gr}}$ be the 
characteristic map with coordinates ordered by $\omega$. Define its associated singular chain by
$$
[\sigma_G,\omega]
:=
(f_{G,\omega})_*(\iota_r),
$$
where $\iota_r$ is the $r$-fold cross product of the identity
singular $1$-simplex of $[0,1]$. It represents the relative
fundamental class of $[0,1]^r$. The assignment
$[G,\omega]\mapsto[\sigma_G,\omega]$ defines a filtered chain map
\begin{equation*}\label{eq:MGr-cellular-comparison}
F_gC
\xrightarrow{\simeq}
C_\bullet^{\mathrm{sing}}
\left(F_gM_{\mathrm{Gr}};\QQ\right),
\qquad
F_gC
:=
\left\langle
[G,\omega]\mid b_1(G)\leq g
\right\rangle,
\end{equation*}
which is a quasi-isomorphism for every $g$
\cite[Definition~5.15 and Theorem~5.18]
{BrownChanGalatiusPayneHopf}.

We now arrive at the associated graded of $C$, the complex that will
connect our bridgeless graph model with graph $K$-theory. Define $F_{-1}C=0$ and set
\begin{equation}\label{eq:GrC-definition}
\operatorname{Gr}_g(C)
:=
F_gC/F_{g-1}C,
\qquad
\operatorname{Gr}(C)
:=
\bigoplus_{g\geq0}\operatorname{Gr}_g(C).
\end{equation}
A graph with $r$ edges and first Betti number $g$ has bidegree
$(g,r-g)$, with total degree $r$.
The differential now retains only the boundary terms that preserve
first Betti number. Contraction does so precisely for nonloop edges,
while deletion does so precisely for bridges. Hence
\begin{equation}\label{eq:GrC-differential}
\partial_{\operatorname{Gr}}[G,\omega]
=
\sum_{\substack{1\leq i\leq r\\e_i\text{ not a loop}}}
(-1)^{i-1}[G/e_i,\omega_i]
-
\sum_{\substack{1\leq i\leq r\\e_i\text{ a bridge}}}
(-1)^{i-1}[G\setminus e_i,\omega_i].
\end{equation}
Geometrically, the omitted faces are exactly those collapsed to
the basepoint in $\operatorname{Gr}_g(M_{\mathrm{Gr}})$.

The cellular comparison induces an isomorphism
\begin{equation*}\label{eq:GrC-GrM-comparison}
\alpha_g:
H_k\bigl(\operatorname{Gr}_g(C)\bigr)
\xrightarrow{\sim}
\widetilde H_k
\bigl(\operatorname{Gr}_g(M_{\mathrm{Gr}})\bigr).
\end{equation*}
Under $\alpha_g$, a graph cycle
$\sum_i a_i[G_i,\omega_i]$ corresponds to the cellular cycle
$\sum_i a_i[\sigma_{G_i},\omega_i]$. Disjoint union defines a product
\begin{equation}\label{eq:GrC-product}
[G,\omega]\,[G',\omega']
=
[G\sqcup G',\omega*\omega'],
\end{equation}
where $\omega*\omega'$ is the concatenated edge ordering.
First Betti numbers and chain degrees add under disjoint union,
so this product respects both gradings. Together with the
coproduct of \cite[Theorem~5.19]{BrownChanGalatiusPayneHopf},
it makes $\operatorname{Gr}(C)$ a bigraded dg bialgebra.

Finally, the filtration on $BK_{\mathrm{Gr}}$ gives the graph
spectral sequence
$$
{}^GE^1_{g,t}
=
H_{g+t}
\bigl(F_gBK_{\mathrm{Gr}},F_{g-1}BK_{\mathrm{Gr}}\bigr)
\cong
H_{g+t}\bigl(\operatorname{Gr}_g(C)\bigr)
\Longrightarrow
H_{g+t}\bigl(BK_{\mathrm{Gr}}\bigr).
$$
The displayed identification uses the rational comparison $q$
and the cellular model above. The filtered map
$BK_{\mathrm{Gr}}\to BK(\ZZ)$ induces a morphism of bialgebra
spectral sequences to the Quillen spectral sequence
\cite[(63) and Corollary~5.20]{BrownChanGalatiusPayneHopf}.
On the first pages, this gives a morphism of bigraded bialgebras
$
\rho:
H_\bullet\bigl(\operatorname{Gr}(C)\bigr)
\longrightarrow
{}^QE^1,
$
whose $(g,t)$-component is
\begin{equation}\label{eq:graph-to-Quillen}
\rho_{g,t}:
H_{g+t}\bigl(\operatorname{Gr}_g(C)\bigr)
\longrightarrow
{}^QE^1_{g,t}.
\end{equation}

Under the geometric description above and the identification
$$
\widetilde H_\bullet
\left(
\left(P_g/\operatorname{GL}_g(\ZZ)\right)\cup\{\infty\}
\right)
\cong
H_\bullet^{\mathrm{BM}}
\left(
P_g/\operatorname{GL}_g(\ZZ)
\right),
$$
this map factors as
\begin{equation}\label{eq:rho-graph-Laplacian}
\rho_{g,t}
=
\theta_{g,t}\circ
(L_g)_*\circ
\alpha_g.
\end{equation}
Thus we compute the image of a graph cycle by taking its cellular
representative, pushing forward along the graph Laplacian map,
and applying the comparison $\theta_{g,t}$ from \eqref{eq: Quillen as PD quadratic forms}.
This factorization is used explicitly in
Section~\ref{sec: Comparing Mg and Ag through compact type}.

\section{The moduli space of curves of compact type}

Throughout this section, fix $g,n\geq 0$ with $2g-2+n>0$. We begin by
describing the boundary of the canonical normal-crossings compactification
$$
\mathcal M_{g,n}^{ct}
\subseteq
\overline{\mathcal M}_{g,n}.
$$

For the full Deligne--Mumford boundary, the associated symmetric dual
complex is naturally identified with the symmetric $\Delta$-complex of
stable graphs; see \cite[$\mathsection$6.1]{ChanGalatiusPayne2022} and
\cite[$\mathsection$7]{AbramovichCaporasoPayne2015}. Our first result identifies the dual complex of the boundary of
$\mathcal M_{g,n}^{ct}\subseteq\overline{\mathcal M}_{g,n}$
with the symmetric $\Delta$-complex of bridgeless stable graphs. Although this identification appears to be well known to experts, we were unable to find a suitable reference in the literature, so we include a proof.

\subsection{The dual complex of the boundary}\label{subsec: dual complex of boundary}We start by recalling the category of 
$\textit{symmetric}~\Delta$-$\textit{complexes}$ (see \cite{ChanGalatiusPayne2021}), i.e., the category  $\mathrm{Fun}(\textbf{FI}^{\mathrm{op}}, \textbf{Set})$, where $\textbf{FI}$ is the category of finite sets with injections. Let $q \geq -1$ be an integer and write
$$
[q]=\{0, \ldots, q\}. 
$$
When $q=-1$, we define the special case $[-1]=\emptyset$. Given any $X: \mathbf{FI}^{\mathrm{op}} \to \mathbf{Set}$ and any integer $q \geq -1$, we write
$
X_q:= X([q])
$
for the set of $\textit{q}$-$\textit{simplices}$ of $X$.
\begin{definition}
    Fix $g,n$ such that $2g-2+n>0$. We define a symmetric $\Delta$-complex
    \begin{equation*}
        \Delta^{bl}_{g,n}: \mathbf{FI}^{\mathrm{op}} \to \mathbf{Set}
    \end{equation*}
    as follows. 
    For each $q\geq -1$, the set $(\Delta^{bl}_{g,n})_q$ is the set of isomorphism classes of pairs $(\Gamma, \omega)$, where
    \begin{enumerate}
        \item $\Gamma$ is a stable $n$-marked graph of genus $g$ having no bridge edges; and 
        \item $\omega: [q] \to E(\Gamma)$ is a bijection, called an $\textit{edge labeling}$.
    \end{enumerate}
    An isomorphism of pairs $(\Gamma, \omega) \to (\Gamma', \omega')$ is a graph isomorphism $\phi: \Gamma \to \Gamma'$ such that $\omega'=\phi \circ \omega$ as maps $[q] \to E(\Gamma')$.
    
    Suppose that $i: [q'] \to [q]$ is an injection. The corresponding face map contracts every edge whose label is not in the image of $i$. The surviving edges inherit labels from $[q']$. Edge contraction is understood in the category of vertex-weighted stable graphs. 
    
    Note that contracting edges in a bridgeless graph produces another bridgeless graph, so these face maps are well defined. 
\end{definition}

\begin{proposition}\label{Boundary indexed by bridgeless graphs}
There is a natural isomorphism of symmetric $\Delta$-complexes

$$
\Delta^{bl}_{g,n}
\cong
\Delta(\Mct\subset\Mbar).
$$

\end{proposition}

\begin{proof}
Let $
B:=\Mbar\setminus\M_{g,n}
$
be the full Deligne--Mumford boundary, and let
$$
D:=\Mbar\setminus\Mct.
$$

If $g=0$, then $\M_{0,n}^{ct}=\Mb_{0,n}$, while a bridgeless stable graph of
genus zero has no edges. Thus both sides have only their
$(-1)$-simplex, and the result is immediate. We may therefore assume
that $g\geq1$.

A stable curve is of compact type if and only if its dual graph is a
tree. Consequently, the boundary divisor $D$ is the
irreducible boundary divisor $\Delta_{\mathrm{irr}}$, that is, the
closure of the locus of curves with a nonseparating node.

By the standard identification of the boundary complex of the
Deligne--Mumford compactification, the $q$-simplices of $\Delta(B)$
are naturally identified with isomorphism classes of pairs
$(\Gamma,\omega)$, where $\Gamma$ is a stable $n$-marked weighted
graph of genus $g$ with $q+1$ edges and
$$
\omega\colon [q]\xrightarrow{\sim}E(\Gamma)
$$
is an edge labeling. Under this identification, the face maps are given
by contraction of the edges whose labels are forgotten; see
\cite[$\mathsection$6.1]{ChanGalatiusPayne2022} and
\cite[$\mathsection$5.2]{ChanGalatiusPayne2021}.

Since $D=\Delta_{\mathrm{irr}}$ is an irreducible component of $B$,
the normalization of $D$ is an open and closed component of the
normalization of $B$. It follows that $\Delta(D)$ is the symmetric
$\Delta$-subcomplex of $\Delta(B)$ consisting of those simplices for
which every chosen local branch belongs to
$\Delta_{\mathrm{irr}}$.

Therefore, it remains to determine which edges of $\Gamma$ correspond
to branches of $\Delta_{\mathrm{irr}}$. To that end, let $C$ be a
curve with dual graph $\Gamma$. By deformation theory of stable
curves, there is an étale neighborhood of $C$ in $\Mbar$ with one
smoothing parameter $t_e$ for each edge $e\in E(\Gamma)$. The
divisor $t_e=0$ is the local branch along which the node corresponding
to $e$ persists. At the generic point of this divisor, every other
node has been smoothed, and hence the dual graph is
$$
\Gamma'
:=
\Gamma/\bigl(E(\Gamma)\setminus\{e\}\bigr);
$$
that is, $\Gamma'$ is obtained by contracting every edge of
$\Gamma$ except $e$; see
\cite[$\mathsection$7.1]{AbramovichCaporasoPayne2015}.

If $e$ is a bridge, then $\Gamma'$ has two vertices joined by a
single edge. The remaining node is therefore separating, so the branch
$t_e=0$ belongs to a separating boundary divisor rather than to
$\Delta_{\mathrm{irr}}$. On the other hand, if $e$ is not a bridge,
then $\Gamma'$ has one vertex and one loop. The remaining node is
therefore nonseparating, so the branch $t_e=0$ belongs to
$\Delta_{\mathrm{irr}}$.

Thus the local branch indexed by $e$ belongs to $D$ if and only if
$e$ is not a bridge. It follows that the simplex represented by
$(\Gamma,\omega)$ belongs to $\Delta(D)$ if and only if $\Gamma$ is
bridgeless.

Under this bijection, the face maps agree. Specifically, forgetting a collection of
branches corresponds, under the identification of the full boundary
complex with stable graphs, to contracting the corresponding edges.
Hence these bijections assemble to an isomorphism of symmetric
$\Delta$-complexes
$$
\Delta^{bl}_{g,n}
\cong
\Delta(D)
=
\Delta(\Mct\subset\Mbar).
$$

\end{proof}

\subsection{The cone complex $M_{g,n}^{ct,\mathrm{trop}}$}

Let $M_{g,n}^{ct,\mathrm{trop}}$ denote the smooth generalized
cone complex associated to the toroidal pair
$
\mathcal M_{g,n}^{ct}
\subseteq
\overline{\mathcal M}_{g,n}.
$

If $g=0$, then $M_{0,n}^{ct,\mathrm{trop}}$ is a point and
$\mathcal M_{0,n}^{ct}=\overline{\mathcal M}_{0,n}$.
Both sides of \eqref{Tropical is weight zero borel more} are
$\QQ$ for $k=0$ and vanish for $k>0$.
We may therefore assume that $g\geq1$ in the argument below.

Applying \cite[Theorem~5.8]{ChanGalatiusPayne2021} to this
compactification, using Poincaré duality and
Proposition~\ref{Boundary indexed by bridgeless graphs}, gives a natural
isomorphism
$$
\left(
W_0H_c^k
\left(
\mathcal M_{g,n}^{ct}
\right)
\right)^\vee
\cong
\widetilde H_{k-1}
\left(
\Delta_{g,n}^{{bl}}
\right).
$$

By \cite[Corollary~5.6]{ChanGalatiusPayne2021}, the link of
$M_{g,n}^{ct,\mathrm{trop}}$ at its cone point is naturally
identified 
with $\lvert\Delta_{g,n}^{{bl}}\rvert$.
Equivalently, $M_{g,n}^{ct,\mathrm{trop}}$ is obtained by
gluing the cones
$
\mathbb R_{\geq0}^{E(\Gamma)}/\operatorname{Aut}(\Gamma)
$
along the faces induced by edge contractions, as $\Gamma$ ranges over
bridgeless stable $n$-marked graphs of genus $g$. Consequently, its
underlying topological space is the unbounded cone on $\lvert\Delta_{g,n}^{{bl}}\rvert$. Its one-point
compactification is therefore the unreduced suspension of
$\lvert\Delta_{g,n}^{{bl}}\rvert$, and hence
$$
\begin{aligned}
H_k^{\mathrm{BM}}
\left(
M_{g,n}^{ct,\mathrm{trop}}
\right)
&\cong
\widetilde H_k
\left(
\Sigma\lvert\Delta_{g,n}^{{bl}}\rvert
\right)\\
&\cong
\widetilde H_{k-1}
\left(
\Delta_{g,n}^{{bl}}
\right).
\end{aligned}
$$

Combining these two identifications yields a natural isomorphism
\begin{equation}\label{Tropical is weight zero borel more}
\kappa_{\mathcal M^{ct}}^{\vee}\colon
\left(
W_0H_c^k
\left(
\mathcal M_{g,n}^{ct}
\right)
\right)^\vee
\overset{\sim}{\longrightarrow}
H_k^{\mathrm{BM}}
\left(
M_{g,n}^{ct,\mathrm{trop}}
\right).
\end{equation}

\begin{remark}
The notation
$M_{g,n}^{ct,\mathrm{trop}}$ refers here to the generalized
cone complex associated to the boundary pair
$
\mathcal M_{g,n}^{ct}
\subset
\overline{\mathcal M}_{g,n}.
$
It should not be confused with the locus in the usual tropical moduli
space $M_{g,n}^{\mathrm{trop}}$ consisting of tropical curves
whose underlying graphs are trees. 
\end{remark}

\section{Weight zero compactly supported cohomology of $\Mct$}

By \eqref{eq: Canonical isomorphism between the kth graded piece and GKbl,k}, the bridgeless Getzler--Kapranov graph complex $\GKzero$ computes the weight zero compactly supported cohomology of $\Mct$. Each vertex $v$ is decorated by an element of
$$
H^0(\overline{\M}_{g_v,n_v})\cong\QQ.
$$

We label each vertex factor with $1$ and suppress these decorations. Thus, $\GKzero$ is spanned by stable bridgeless $n$-marked graphs $\Gamma$ of genus $g$, together with a total ordering of $E(\Gamma)$, subject to the relation that changing the edge ordering by a permutation multiplies the generator by the sign of the permutation. Consequently, any graph admitting an automorphism that induces an odd permutation of its edges represents zero in $\GKzero$.

\subsection{The weight zero differential} The differential $\delta = \deltasp+\deltal$ is a sum of local contributions at the vertices. Since all decorations on the vertices in the weight zero bridgeless complex are trivial, it amounts to splitting vertices and attaching loops at vertices in every allowable way. In other words, the differential acts just like the differential on the standard commutative graph complex with the exception that we only allow a vertex to split into a non-bridge edge. The contribution $\deltasp$ at a vertex $v$ of genus $g(v)=h$ is represented as
\begin{equation*}\label{eq:delta-split-undecorated}
\deltasp\colon\quad
\begin{tikzpicture}[baseline=-.65ex]
    \node[gkvertex] (v) at (0,0) {$\scriptstyle h$};
    \draw
        (v) edge +(-.3,-.3)
            edge +(-.3,0)
            edge +(-.3,.3)
            edge +(.3,-.3)
            edge +(.3,0)
            edge +(.3,.3);
\end{tikzpicture}
\longmapsto
\sum_{a+b=h}
\sum
\begin{tikzpicture}[baseline=-.65ex]
    \node[gkvertex] (v) at (0,0) {$\scriptstyle a$};
    \node[gkvertex] (w) at (0.5,0) {$\scriptstyle b$};
    \draw
        (v) edge (w)
        (v) edge +(-.3,-.3)
            edge +(-.3,0)
            edge +(-.3,.3)
        (w) edge +(.3,-.3)
            edge +(.3,0)
            edge +(.3,.3);
\end{tikzpicture}.
\end{equation*}

 For each decomposition $h=a+b$, the second sum ranges over all ways of distributing the half-edges incident to $v$ between the two new vertices for which both vertices are stable and the resulting graph is bridgeless. Similarly, the contribution of $\deltal$ at $v$ is represented by
\begin{equation*}\label{eq:delta-loop-undecorated}
\deltal\colon\quad
\begin{tikzpicture}[baseline=-.65ex]
    \node[gkvertex] (v) at (0,0) {$\scriptstyle h$};
    \draw
        (v) edge +(-.5,-.5)
            edge +(0,-.5)
            edge +(.5,-.5);
\end{tikzpicture}
\longmapsto
\begin{tikzpicture}[baseline=-.65ex]
    \node[gkvertex] (v) at (0,0) {$\scriptscriptstyle h-1$};
    \draw
        (v) edge[loop above,out=135, in=45,looseness=5,min distance=12mm] (v)
        (v) edge +(-.5,-.5)
            edge +(0,-.5)
            edge +(.5,-.5);
\end{tikzpicture},
\qquad h>0.
\end{equation*}
Here, the right hand side is zero if $h=0$. In both formulas, the newly created edge is placed first in the edge ordering, while the relative ordering of the remaining edges is unchanged. We now pass to a smaller quotient.

\begin{definition}
    Write
$\mathrm{GK}_{g,n}^{bl}:=
\bigoplus_k\mathrm{GK}_{g,n}^{bl,k}$ and let $I_{g,n}^{bl} \subset \mathrm{GK}_{g,n}^{bl}$ be the graded subspace spanned by graphs containing at least one vertex of positive genus. The \textit{simplified bridgeless Getzler--Kapranov graph complex} is defined to be the quotient
    \begin{equation}\label{GK bar}
        \overline{\mathrm{GK}}_{g,n}^{bl}= \mathrm{GK}_{g,n}^{bl}/ (I_{g,n}^{bl} + \delta I_{g,n}^{bl}).  \end{equation}
\end{definition}

In weight zero, $\GKbarzero$ is spanned by bridgeless stable graphs with only genus zero vertices and no loop edges. 

For the corresponding complexes in which bridges are allowed, the quotient map  
$$
\mathrm{GK}_{g,n}^{0}
\longrightarrow
\overline{\mathrm{GK}}_{g,n}^{0}
$$
is a quasi-isomorphism for every stable pair
$(g,n)$; see \cite{ChanGalatiusPayne2021,ChanGalatiusPayne2022} and \cite[$\mathsection$3.5]{PayneWillwacherWeight2}.
Theorem~\ref{thm:intro-weight zero-model} establishes the analogous statement for the bridgeless complex.

The heart of the proofs of
Theorems~\ref{thm:intro-weight zero-model} and
\ref{thm:intro-weight two-model}
begins with a filtration by the number of vertices. On the associated graded complex, only $\deltal$ survives. We compute its cohomology by decomposing it into subcomplexes indexed by loopless cores, each built from tensor products of local complexes attached to the individual vertices. These single-vertex complexes were introduced in
\cite[$\mathsection$4.1]{PayneWillwacherWeight2}. We recall them next and then adapt the filtration argument of
\cite[$\mathsection$4.3]{PayneWillwacherWeight2}
to the bridgeless setting.

\subsection{The single-vertex complex}

The single-vertex complex isolates the local action of $\deltal$ at one vertex while retaining the symmetries among the loops created there. Following \cite[(15)]{PayneWillwacherWeight2}, define
\begin{equation}\label{eq:single-vertex-complex}
V_{g,n}
:=
\bigoplus_{h=0}^{g}
\left(
H^\bullet(\Mb_{g-h,n+2h})
\otimes
\QQ[-1]^{\otimes h}
\right)_{\mathbb S_2\wr\mathbb S_h}.
\end{equation}
Here the last $2h$ marked points are grouped into $h$ pairs. Writing 
$
\mathbb S_2\wr\mathbb S_h
=
\mathbb S_2^h\rtimes\mathbb S_h,
$ the factors of $\mathbb S_2^h$ exchange the two markings within each pair, while $\mathbb S_h$ permutes the pairs and the corresponding factors of $\QQ[-1]$, with signs determined by the Koszul sign rule.

The differential on $V_{g,n}$ is induced by $\deltal$. The cohomological degree of the factor
$H^\bullet(\Mb_{g-h,n+2h})$ defines a weight grading preserved by the differential:

$$
V_{g,n}
=
\bigoplus_k V_{g,n}^k.
$$

\subsubsection{The weight zero single-vertex complex}

In weight zero, this becomes
$$
V_{g,n}^0
=
\bigoplus_{h=0}^{g}
\left(
H^0(\Mb_{g-h,n+2h})
\otimes
\QQ[-1]^{\otimes h}
\right)_{\mathbb S_2\wr\mathbb S_h}.
$$
For $h\geq2$, a transposition in $\mathbb S_h$ exchanges two odd factors of $\QQ[-1]$ and therefore acts by $-1$. The same transposition acts trivially on
$H^0(\Mb_{g-h,n+2h})\cong\QQ$. Hence the corresponding coinvariant space vanishes. The summand indexed by $h$ is therefore
$$
\left(
H^0(\Mb_{g-h,n+2h})
\otimes
\QQ[-1]^{\otimes h}
\right)_{\mathbb S_2\wr\mathbb S_h}
\cong
\begin{cases}
\QQ, & h=0,\\
\QQ[-1], & h=1,\\
0, & h\geq2.
\end{cases}
$$
When $g \geq 1$, the differential from the $h=0$ to the $h=1$ summand is non-zero. Thus the differential is an isomorphism. When $g=0$, the summands with $h \geq 1$ vanish. We conclude that
\begin{equation}\label{eq:weight zero-single-vertex-cohomology}
H(V_{g,n}^0)
\cong
\begin{cases}
\QQ, & g=0 \text{ and } n\geq3,\\
0, & \text{otherwise},
\end{cases}
\end{equation}
as in \cite[(16)]{PayneWillwacherWeight2}.

With this local calculation in hand, we now prove
Theorem~\ref{thm:intro-weight zero-model}.

\begin{proof}[Proof of Theorem~\ref{thm:intro-weight zero-model}]
Let
$
q\colon\GKzero\longrightarrow\GKbarzero
$
denote the quotient map. We filter both complexes by the number of vertices. Let
$\F^p\GKzero$ and $\F^p\GKbarzero$ be the subcomplexes spanned by graphs with at least $p$ vertices. These are finite decreasing filtrations, and $q$ is a filtered map. 

The differential $\deltal$ preserves the number of vertices, whereas
$\deltasp$ increases it by one. It follows that the associated graded complexes are
$$
\operatorname{gr}_{\F}\GKzero
\cong
(\GKzero,\deltal)
\qquad\text{and}\qquad
\operatorname{gr}_{\F}\GKbarzero
\cong
(\GKbarzero,0).
$$

We compute the cohomology of the first complex by organizing its graphs according to their loopless cores. The
\emph{loopless core} of a stable graph $G$ is obtained by contracting every loop edge and increasing the genus of its incident vertex by one for each contracted loop. This operation preserves the total genus, stability, and bridgelessness. Notice that the differential $\deltal$ does not change the loopless core of a graph.

For a loopless stable bridgeless graph $\Gamma$, let
$$
(\mathrm{GK}_{\Gamma}^{{bl},0},\deltal)
\subseteq
(\GKzero,\deltal)
$$
be the subcomplex spanned by graphs with loopless core $\Gamma$. We obtain a decomposition
$$
(\GKzero,\deltal)
=
\bigoplus_{[\Gamma]}
(\mathrm{GK}_{\Gamma}^{{bl},0},\deltal),
$$
where the sum runs over isomorphism classes of such graphs.

Fix a loopless core $\Gamma$. For each $v\in V(\Gamma)$, let $g_v$ be its genus and let $n_v$ be the number of half-edges incident to it. Since $\deltal$ acts independently at each vertex, there is an isomorphism
$$
(\mathrm{GK}_{\Gamma}^{{bl},0},\deltal)
\cong
\bigg(
\bigotimes_{v\in V(\Gamma)}
V_{g_v,n_v}^0
\otimes
\QQ[-1]^{\otimes|E(\Gamma)|}
\bigg)_{\operatorname{Aut}(\Gamma)}.
$$

The K\"unneth formula, together with the fact that taking coinvariants with respect to finite group actions commutes with taking cohomology, gives

$$
H(\mathrm{GK}_{\Gamma}^{{bl},0},\deltal)
\cong
\bigg(
\bigotimes_{v\in V(\Gamma)}
H(V_{g_v,n_v}^0)
\otimes
\QQ[-1]^{\otimes |E(\Gamma)|}\bigg)
_{\operatorname{Aut}(\Gamma)}.
$$

By \eqref{eq:weight zero-single-vertex-cohomology}, this space vanishes whenever $\Gamma$ contains a vertex of positive genus. If every vertex has genus zero, stability gives $n_v\geq3$, and each factor
$H(V_{0,n_v}^0)\cong\QQ$ is represented by the summand with no local loops. Consequently,
$$
H(\GKzero,\deltal)
\cong
\bigoplus_{\substack{[\Gamma]\\g_v=0\text{ for all }v}}
\big(\QQ[-1]^{\otimes|E(\Gamma)|}\big)
_{\operatorname{Aut}(\Gamma)}
\cong
\GKbarzero
$$
as graded vector spaces. 

The final isomorphism is induced by $q$. In particular, for a loopless core $\Gamma$ whose vertices all have genus zero, the surviving class is represented by $\Gamma$ itself, which $q$ sends to the corresponding generator of $\GKbarzero$. Thus $q$ induces an isomorphism on the $E_1$-pages.
$$
\bigl(H(\GKzero,\deltal),\deltasp\bigr)
\cong
(\GKbarzero,\deltasp).
$$
Since the filtrations are finite, the spectral-sequence comparison theorem \cite[Theorem~5.2.12]{WeibelHomologicalAlgebra} shows that $q$ is a quasi-isomorphism. 
\end{proof}

\begin{remark}
There is another useful way to think about the complex above. When
$n=0$, the graded dual of
$\overline{\mathrm{GK}}_{g,0}^{{bl},0}$ identifies, after
translating the grading and orientation conventions, with the
loop-order-$g$ summand of the even simple bridgeless graph complex
studied by Willwacher in \cite{WillwacherBridgeless}. Thus Willwacher
studies essentially the same complex in the dual convention. His
edge-contraction differential is dual to the edge-expansion differential
used here.

His alternative description begins with a simple observation. If one
cuts a connected graph at all of its bridges, the remaining bridgeless
components are arranged along a tree. In operadic language, this
identifies the ordinary graph complex with external legs as the cyclic
bar construction on its bridgeless cyclic suboperad.

Applying the cyclic cobar construction to this description produces
a complex whose elements may be pictured as trees with vertices
decorated by graphs from the ordinary graph complex with external legs;
the decorating graphs need not themselves be bridgeless. Its differential
has two parts. One applies the usual graph differential within each
decoration, while the other splits vertices of the indexing tree.
Willwacher proves that this decorated-tree complex is quasi-isomorphic
to the bridgeless graph complex. Specializing to arity zero gives the
usual bridgeless graph complex, and the same construction applies to
the simple variant considered here; see
\cite[(3), Theorem~1.1, and Theorem~3.7]{WillwacherBridgeless}.

We independently computed the homology of the graded-dual simple
bridgeless complex through genus $7$. After translating the grading
conventions, our dimensions agree with the even case of
\cite[Figure~1]{WillwacherBridgeless}. 
\end{remark}

We return to the additional structure carried by weight zero cohomology in Section~\ref{sec: Comparing Mg and Ag through compact type}. For now, we turn to the next nontrivial weight-graded piece,
$\operatorname{gr}_2^W H_c^\bullet(\Mct)$.

\section{Weight two compactly supported cohomology of $\Mct$}\label{sec: weight two compactly supported cohomology of Mct}

We now study the second graded piece of the compactly supported cohomology of $\Mct$. By
\eqref{eq: Canonical isomorphism between the kth graded piece and GKbl,k}, there is a natural isomorphism
$$
\operatorname{gr}_2^W
H_c^\bullet(\Mct)
\cong
H^\bullet(\GKtwo).
$$

A generator of $\GKtwo$ consists of a stable bridgeless $n$-marked graph $\Gamma$ of genus $g$, together with a decoration at each vertex $v$,
$$
\alpha_v\in H^{k_v}(\Mb_{g_v,n_v}),
\qquad
\sum_{v\in V(\Gamma)}k_v=2.
$$
It also carries an orientation of its structural edges, represented by a total ordering of $E(\Gamma)$. As usual, permuting this ordering by $\sigma$ multiplies the generator by $\operatorname{sgn}(\sigma)$. Here the structural edges are the internal edges of $\Gamma$, as opposed to its $n$ marked legs.

Since
$
H^1(\Mb_{h,m})=0
$
for every stable pair $(h,m)$, there is only one possible distribution of the decoration degrees. Precisely one vertex is decorated by a class in
$H^2(\Mb_{g_v,n_v})$, while all remaining vertices are decorated by $1$ in $H^0$. We call this distinguished vertex the \emph{special vertex}. Thus, a generator of $\GKtwo$ may be represented by a stable bridgeless $n$-marked graph with one special vertex carrying an $H^2$-decoration. In our graphical notation, we omit the ordering of the structural edges and indicate the special vertex by a double circle.

To describe the possible decorations at the special vertex, we recall the presentation of
$H^2(\overline{\mathcal M}_{g,n})$ given in
\cite[Theorem~2.2]{ArbarelloCornalba1998}; see also \cite[$\mathsection$3.1]{PayneWillwacherWeight2}. 

Set $[n]=\{1,\ldots,n\}$. The group
$H^2(\overline{\mathcal M}_{g,n})$ is generated by
$$
\kappa_1,\qquad
\psi_1,\ldots,\psi_n,\qquad
\delta_{{irr}},\qquad
\delta_{a,A},
$$
where $A\subseteq[n]$, $0\leq a\leq g$, 
$
2a-2+|A|\geq0$, and 
$2(g-a)-2+|A^c|\geq0.
$
The boundary classes satisfy the symmetry relation
$
\delta_{a,A}=\delta_{g-a,A^c}.
$
Write

$$
\psi:=\sum_{i=1}^n\psi_i,
\qquad
\delta_a:=\sum_A\delta_{a,A},
$$
where the second sum ranges over the distinct stable boundary divisors. When
$g=2a$, each class
$\delta_{a,A}=\delta_{a,A^c}$ is counted only once. Modulo the symmetry
relations above, a generating set for the remaining relations is as follows.

\begin{enumerate}
\item If $g\geq3$, there are no further relations.

\item If $g=2$, there is one additional relation:
$$
5\kappa_1
=
5\psi+\delta_{{irr}}-5\delta_0+7\delta_1.
$$

\item If $g=1$, the additional relations are
$$
\kappa_1=\psi-\delta_0,
$$

$$
12\psi_p
=
\delta_{{irr}}
+
12\sum_{\substack{S\ni p\\ |S|\geq2}}\delta_{0,S}.
$$

\item If $g=0$, then $\delta_{{irr}}=0$, and for every
choice of distinct $x,y,z\in[n]$,
$$
\kappa_1
=
\sum_{A \not\ni x,y}
(|A|-1)\delta_{0,A},
\qquad
\psi_z
=
\sum_{z \in A \not\ni x,y}
\delta_{0,A}.
$$
Here $z$ and $x \neq y$ run over elements of $\{1,\ldots,n\}$.

\end{enumerate}

For our later combinatorial calculations, the following equivalent presentation of the genus zero case will be more convenient. 

\begin{lemma}~\cite[Lemma~3.4]{PayneWillwacherWeight2}\label{lem:psi and boundary rep for H^2 M0,n}
    The cohomology group $H^2(\Mb_{0,n})$ is generated by boundary classes $\delta_{0,A}$ for subsets $A\subset \{1,\ldots, n\}$ with $2 \leq |A| \leq n-2$, and the $\psi$-classes $\psi_1,\ldots, \psi_n$ modulo the following relations
    \begin{enumerate}
        \item $\delta_{0,A}=\delta_{0,A^c}$,
        \item  $\psi_i+\psi_j=\sum_{i\in A, j \in A^c}\delta_{0,A}$.
    \end{enumerate}
\end{lemma}

The differential
$\delta=\deltasp+\deltal$ on $\GKtwo$ is the sum of its local contributions at the vertices of a generator. At a nonspecial vertex, nothing new occurs. The differential agrees with the weight zero
differential, splitting the vertex and creating loops in all allowable ways.
At the special vertex, the $H^2$-decoration must also be pulled back along
the corresponding gluing or clutching morphism. We now recall the maps and
pullback formulas needed to describe these contributions.

It will be convenient to index marked points by arbitrary finite sets.
For a finite set $S$, write $\M_{g,S}$ for the moduli space of smooth
genus $g$ curves whose markings are indexed by $S$, and
$\Mb_{g,S}$ for its Deligne--Mumford compactification. Let
$A\subseteq S$, set $A^c:=S\setminus A$, and choose new labels
$q,r,t,t'\notin S$. Fix a stable
$(A^c\cup\{r\})$-marked curve $C_0$ of genus $h-a$. Gluing $C_0$
to the varying curve by identifying the markings $q$ and $r$ defines
\begin{equation}\label{eq: theta gluing map}
\theta\colon
\Mb_{a,A\cup\{q\}}
\longrightarrow
\Mb_{h,S},
\qquad
C\longmapsto C\mathop{\cup}_{q\sim r}C_0.
\end{equation}

The induced pullback $\theta^*$ on cohomology is independent of the choice
of $C_0$. Similarly, identifying the two markings $t$ and $t'$ defines
the self-gluing morphism
\begin{equation}\label{eq: xi clutching map}
\xi\colon
\Mb_{h-1,S\cup\{t,t'\}}
\longrightarrow
\Mb_{h,S}.
\end{equation}

The map $\theta$ governs the contribution of $\deltasp$ at the special
vertex, while $\xi$ governs the contribution of $\deltal$. The relevant
pullback formulas are recorded in
\cite[Lemmas~3.2 and~3.3]{ArbarelloCornalba1998}.

We begin with
$\theta^*$: 
\begin{equation}\label{eq: theta^* for k_1, deltairr, psi}
\theta^*(\kappa_1)=\kappa_1,
\qquad
\theta^*(\delta_{irr})=\delta_{irr}, \qquad \theta^*(\psi_i)
=
\begin{cases}
\psi_i & \text{if } i\in A,\\
0      & \text{if } i\in A^c.
\end{cases}
\end{equation}

It remains to describe the pullback of the boundary classes
$\delta_{b,B}$. First
suppose that $A=S$. Then
\begin{equation}\label{eq: theta^* boundary sc A=S}
\theta^*(\delta_{b,B})
=
\begin{cases}
\delta_{2a-h,S\cup\{q\}}-\psi_q
&
\text{if }(b,B)\in
\bigl\{(a,S),(h-a,\emptyset)\bigr\},
 \\
\delta_{b,B}+\delta_{b+a-h,B\cup\{q\}}
&
\text{otherwise}.
\end{cases}
\end{equation}

If $A\neq S$, then
\begin{equation}\label{eq: theta^* for boundary A neq S}
\theta^*(\delta_{b,B})
=
\begin{cases}
-\psi_q
&
\text{if }(b,B)\in
\bigl\{(a,A),(h-a,A^c)\bigr\},\\
\delta_{b,B}
&
\text{if }B\subseteq A
\text{ and }(b,B)\neq(a,A),\\
\delta_{b+a-h,(B\setminus A^c)\cup\{q\}}
&
\text{if }A^c\subseteq B
\text{ and }(b,B)\neq(h-a,A^c),\\
0
&
\text{otherwise}.
\end{cases}
\end{equation}

For the self-gluing morphism $\xi$, the pullback formulas are

\begin{equation}\label{eq: pullback for clutching map xi} 
\begin{aligned}
    \xi^*(\kappa_1)&=\kappa_1,\\
    \xi^*(\psi_i)&=\psi_i,\\
    \xi^*(\delta_{irr})
&=
\delta_{irr}
-\psi_t-\psi_{t'}
+
\sum_{
t\in B,\;t'\notin B}
\delta_{b,B},\\
\xi^*(\delta_{a,A})
&=
\begin{cases}
\delta_{a,A}
&
\text{if }h=2a\text{ and }A=S=\emptyset,\\
\delta_{a,A}
+\delta_{a-1,A\cup\{t,t'\}}
&
\text{otherwise}.
\end{cases}
\end{aligned}
\end{equation}

\subsection{The graphical differential}
We use the graphical conventions of
\cite[$\mathsection\mathsection$3.2--3.3]{PayneWillwacherWeight2}.
We draw the special vertex with a double circle. If the decoration is
$\kappa_1$ or $\delta_{irr}$, we simply write the class next to the special vertex. If the decoration is $\psi_i$, we indicate this decoration by an arrow on
the corresponding $i$-th half-edge. A boundary decoration
$\delta_{a,A}$ is represented by replacing the special vertex by
vertices of genera $a$ and $h-a$, joined by a distinguished edge,
with the half-edges in $A$ and $A^c$ attached to the respective
vertices. The distinguished edge records the decoration and is not a
structural edge.

With these conventions, the pullback formulas above translate directly into a graphical differential. For example, the contribution of
$\deltasp$ to a boundary-decorated vertex is
\begin{align*}
    \deltasp: \graphab \mapsto &\sum_{b'+b''=b} \sum  \graphabsl + \sum_{a'+a''=a} \sum\graphabsr\\ &-\graphabpsil -\graphabpsir.
\end{align*}
In each of the first two terms, the inner sum ranges over all distributions of the incident half-edges for which the resulting vertices are stable. The remaining graphical formulas for
$\deltasp$ and $\deltal$ are given in
\cite[$\mathsection$3.3]{PayneWillwacherWeight2}. The same local
formulas apply here, except that we retain only those
summands whose underlying structural graph is bridgeless.

We next introduce two auxiliary graph complexes that isolate the two contributions appearing in the cohomology of $\GKtwo$.

\subsection{The reduced weight two graph complex}

Taking the weight two part of the quotient \eqref{GK bar}, we obtain the
\textit{reduced weight two bridgeless Getzler--Kapranov graph complex}
$$
\GKbartwo
=
\frac{\GKtwo}
{I_{g,n}^{{bl},2}+\delta I_{g,n}^{{bl},2}}.
$$

The complex is generated by bridgeless graphs whose vertices all have genus zero, with no loops at nonspecial vertices. Loops at the special
vertex are allowed, subject to the following additional relation. If
$s$ and $t$ are the two half-edges of such a loop, then
$\delta_{0,A}$ vanishes whenever
$$
\{s,t\}\subseteq A
\qquad\text{or}\qquad
\{s,t\}\subseteq A^c.
$$
These two cases are interchanged by the symmetry
$\delta_{0,A}=\delta_{0,A^c}$. 

\subsection{The $\Delta^{\mathrm{irr}}_{g,n}$-complex}
 Let
$$
\Delta_{g,n}^{\mathrm{irr}}\subset\GKtwo
$$
be the graded subspace spanned by bridgeless graphs with no loops, whose special
vertex has genus one and is decorated by $\delta_{irr}$, and
whose remaining vertices all have genus zero. 
\begin{lemma}\label{lem:Delta_irr is a dg subspace}
The graded subspace
$\Delta_{g,n}^{\mathrm{irr}}\subset\GKtwo$
 is a dg subspace. Its differential is the restriction of $\deltasp$.
\end{lemma}

\begin{proof}
Let $\Gamma\in\Delta_{g,n}^{\mathrm{irr}}$ be a generator. At a nonspecial vertex,
$\deltal$ vanishes because the vertex has genus zero, while every
nonzero term of $\deltasp$ again lies in
$\Delta_{g,n}^{\mathrm{irr}}$.

Now consider $\deltasp$ at the special vertex. Splitting a genus one
vertex distributes its genus as $0+1$. Since
$\delta_{irr}$ vanishes on every genus zero moduli space,
only the term in which the genus one vertex retains the
$\delta_{irr}$-decoration survives. Graphically,
$$
\graphdeltairrone
\ \overset{\deltasp}{\longmapsto}\
\sum \graphsplitonezerodeltairr
\ \in\Delta_{g,n}^{\mathrm{irr}},
$$
where the sum ranges over all distributions of the incident half-edges
for which the resulting graph is stable and bridgeless.

It remains to compute $\deltal$ at the special vertex. Specializing
\eqref{eq: pullback for clutching map xi} to
$$
\xi\colon
\overline{\mathcal M}_{0,S\cup\{t,t'\}}
\longrightarrow
\overline{\mathcal M}_{1,S},
$$
we obtain
$$
\xi^*(\delta_{irr})
=
-\psi_t-\psi_{t'}
+
\sum_{\substack{t\in A\\t'\in A^c}}\delta_{0,A}.
$$
By Lemma \ref{lem:psi and boundary rep for H^2 M0,n},
$$
\psi_t+\psi_{t'}
=
\sum_{\substack{t\in A\\t'\in A^c}}\delta_{0,A}.
$$
Therefore $(\deltasp+\deltal)(\Gamma)=\deltasp(\Gamma)\in \Delta^{\mathrm{irr}}_{g,n}$.
\end{proof}

For the corresponding graph complexes in which bridges are allowed, \cite[Theorem~3.3]{PayneWillwacherWeight2} proves that the
quotient map
$$
\mathrm{GK}_{g,n}^2
\longrightarrow
\overline{\mathrm{GK}}_{g,n}^2
$$
is a quasi-isomorphism for every stable pair $(g,n)\neq(1,1)$. These complexes compute $\mathrm{gr}_2^W(H^\bullet_c(\M_{g,n}))$. A key difference appears
for compact type. Specifically, the positive-genus contribution
$\Delta_{g,n}^{\mathrm{irr}}$ survives and must be retained alongside
the reduced complex. This is the content of Theorem~\ref{thm:intro-weight two-model}, and we now
assemble the ingredients needed for its proof.

\subsection{The weight two single-vertex complex}

Taking the weight two summand of \eqref{eq:single-vertex-complex} gives
\begin{equation*}
    0 \to H^2(\Mb_{g,n}) \to H^2(\Mb_{g-1,n+2})_{\mathbb{S}_2} \to H^2(\Mb_{g-2,n+4})_{\mathbb{S}_2 \wr \mathbb{S}_2} \to H^2(\Mb_{g-3,n+6})_{\mathbb{S}_2 \wr \mathbb{S}_3} \to \cdots.
\end{equation*}
Here the action of $\mathbb{S}_2 \wr \mathbb{S}_h$ is understood to include a sign that accounts for the action on the tensor factors $\QQ[-1]^{\otimes h}$ via the Koszul sign rule.

Graphically, the $h$-th summand is
represented by a single vertex of genus $g-h$, decorated by a class in
$H^2(\Mb_{g-h,n+2h})$, with $n$ marked legs and $h$ loops. An
ordering of the loops determines an orientation, with reordering governed
by the usual sign relation. The differential $\deltal$ adds one loop and
pulls back the vertex decoration along the corresponding clutching map.

\begin{proposition}~\cite[Proposition~4.2]{PayneWillwacherWeight2}\label{prop: coh of single vertex complex weight 2}
For $g\geq2$,
$$
H(V_{g,n}^2)=0.
$$
For $g=1$,
$$
H(V_{1,n}^2)
\cong
\QQ\delta_{irr}
\oplus
\frac{
H^2(\Mb_{0,n+2})_{\mathbb S_2}
}{
\xi^*H^2(\Mb_{1,n})
}.
$$
For $g=0$, the differential on $V_{0,n}^2$ vanishes.
\end{proposition}

\begin{proof}
We recall the genus one proof, which will be used below, and refer to
\cite[Proposition~4.2]{PayneWillwacherWeight2} for the remaining cases.
For $g=1$, the single-vertex complex is the two-term complex
$$
H^2(\Mb_{1,n})
\xrightarrow{\xi^*}
H^2(\Mb_{0,n+2})_{\mathbb S_2}.
$$

By the genus one case of
\cite[Theorem~2.2]{ArbarelloCornalba1998},
$H^2(\Mb_{1,n})$ is freely generated by
$\delta_{irr}$ and the boundary classes
$\delta_{0,A}$, where
$A\subseteq\{1,\ldots,n\}$ and $|A|\geq2$.
The pullback formulas give

$$
\xi^*(\delta_{irr})
=
-\psi_t-\psi_{t'}
+
\sum_{\substack{t\in A\\t'\in A^c}}\delta_{0,A}
=
0,
$$
where the final equality follows from
Lemma~\ref{lem:psi and boundary rep for H^2 M0,n}. Moreover,
$$
\xi^*(\delta_{0,A})=\delta_{0,A}.
$$

By \cite[Lemma~4.4]{PayneWillwacherWeight2}, these latter images are
linearly independent in
$H^2(\Mb_{0,n+2})_{\mathbb S_2}$. It follows that
$$
\ker(\xi^*)=\QQ\delta_{irr},
\qquad
\operatorname{coker}(\xi^*)
=
\frac{
H^2(\Mb_{0,n+2})_{\mathbb S_2}
}{
\xi^*H^2(\Mb_{1,n})
}.
$$

These are the two cohomology groups of the two-term complex, giving the
claimed decomposition of $H(V_{1,n}^2)$.
\end{proof}

\subsection{The weight two reduction}

The proof of Theorem~\ref{thm:intro-weight two-model} uses the
filtration by the number of vertices to construct a spectral sequence
whose $E_1$-page is concentrated in two adjacent rows, identified with
$\Delta_{g,n}^{\mathrm{irr}}$ and $\GKbartwo$.
The only possible higher differential is $d_2$, which connects these
rows. The following two lemmas provide the ingredients needed to
eliminate this obstruction.

\begin{definition}
Let $G\in\GKtwo$ have special vertex $w$. Contract every loop of $G$, increasing the genus of its incident vertex by one for each loop contracted. Let $\Gamma$ be the resulting loopless graph. We call the pair 
$
(\Gamma,w)
$
the \emph{pointed loopless core} of $G$.
\end{definition}

By construction, $\deltal$ preserves the pointed loopless core. Following the loopless-core decomposition in the proof of
\cite[Theorem~3.3]{PayneWillwacherWeight2}, we refine each core summand by
retaining the position of the special vertex. For a pointed loopless core
$(\Gamma,w)$, let
$$
C_{(\Gamma,w)}
\subseteq
(\GKtwo,\deltal)
$$
be the subcomplex spanned by generators with pointed loopless core
isomorphic to $(\Gamma,w)$. We then have
\begin{equation}\label{eq: pointed lc decomp}
(\GKtwo,\deltal)
=
\bigoplus_{[(\Gamma,w)]}
C_{(\Gamma,w)}.
\end{equation}

For a fixed pointed core, $\deltal$ acts independently at each vertex of $\Gamma$. Consequently,
\begin{equation}\label{eq: C=VtimesV}
C_{(\Gamma,w)}
\cong
\bigg(
V^2_{g(w),n(w)}
\otimes
\bigotimes_{\substack{v\in V(\Gamma)\\v\neq w}}
V^0_{g(v),n(v)}
\otimes
\QQ[-1]^{\otimes |E(\Gamma)|}
\bigg)_{\operatorname{Aut}(\Gamma,w)},
\end{equation}
where $\mathrm{Aut}(\Gamma, w)$ denotes the subgroup of automorphisms of $\Gamma$ that fixes $w$.

\begin{lemma}\label{lem:loopless-genus one-kernel}
Let $L_{g,n}\subset\GKtwo$ be the graded subspace spanned by loopless
graphs having exactly one genus one vertex, with all remaining vertices of
genus zero. Then
$$
\ker\!\Big(
L_{g,n}
\xrightarrow{\ \deltal\ }
\GKtwo
\Big)
=
\Delta_{g,n}^{\mathrm{irr}}.
$$
\end{lemma}

\begin{proof}
Lemma~\ref{lem:Delta_irr is a dg subspace} shows that $\deltal$
vanishes on $\Delta_{g,n}^{\mathrm{irr}}$, giving one containment. Conversely, let $z\in L_{g,n}$ satisfy
$\deltal z=0$. Decompose $z$ according to its pointed loopless core:
$$
z
=
\sum_{[(\Gamma,w)]}
z_{(\Gamma,w)},
\qquad
z_{(\Gamma,w)}
\in
L_{g,n}\cap C_{(\Gamma,w)}.
$$
Since $\deltal$ preserves the pointed loopless core, each $z_{(\Gamma, w)}$ satisfies $\deltal(z_{(\Gamma,w)})=0$. We show that each of these components belongs to $\Delta_{g,n}^{\mathrm{irr}}$.

Fix $(\Gamma,w)$, and let $u$ be its unique genus one vertex.
All other vertices have genus zero, so only $u$ contributes
to $\deltal$. We may compute the kernel on its decoration,
since tensoring with the remaining vertex decorations and then taking finite-group
coinvariants over $\QQ$ preserves kernels.

If $u\neq w$, the vertex $u$ is decorated by $1\in H^0$,
and the local pullback sends $1$ to $1$.
Thus $\deltal$ is
injective on the subspace of loopless generators with this
pointed core. It follows that $z_{(\Gamma,w)}=0$.

If $u=w$, the local map on the decoration is
$$
\xi^*\colon
H^2(\Mb_{1,n(w)})
\longrightarrow
H^2(\Mb_{0,n(w)+2})_{\mathbb S_2}.
$$
By the genus one calculation in
Proposition~\ref{prop: coh of single vertex complex weight 2},
its kernel is $\QQ\delta_{irr}$. Hence $z_{(\Gamma,w)}$ is represented by graphs whose special
vertex is decorated by $\delta_{irr}$, and therefore
belongs to $\Delta_{g,n}^{\mathrm{irr}}$.

Thus every component of $z$ belongs to
$\Delta_{g,n}^{\mathrm{irr}}$, so $z$ does as well.
This proves the reverse containment.
\end{proof}

The differential $\deltasp$ preserves both
$L_{g,n}$ and $\Delta_{g,n}^{\mathrm{irr}}$, which we henceforth regard as complexes with $\deltasp$. As always, only splittings with bridgeless output are retained.

\begin{lemma}\label{lem:corrected-rho}
The inclusion
$
\iota\colon
(\Delta_{g,n}^{\mathrm{irr}},\deltasp)
\hookrightarrow
(L_{g,n},\deltasp)
$
admits a retraction as a map of complexes. 
\end{lemma}

\begin{proof}

For every nonempty finite set $S$, the classes 
$$
\deltairr,
\qquad
\delta_{0,A}
\quad
(A\subseteq S,\ |A|\geq2)
$$
 form a basis of $H^2(\Mb_{1,S})$; see \cite[Proof of Proposition~4.2, case $g=1$]{PayneWillwacherWeight2}. Define
$
\lambda_S\colon H^2(\Mb_{1,S})\longrightarrow\QQ
$
by
$$
\lambda_S(\deltairr)=1,
\qquad
\lambda_S(\delta_{0,A})=
\begin{cases}
-\dfrac{1}{12},&A=S,\\
~~~~0,&A\subsetneq S.
\end{cases}
$$
When $|S|=1$, the second formula is vacuous.

Let $u$ and $w$ denote the unique genus one vertex and the
special vertex of a generator of $L_{g,n}$, respectively. If
$u\neq w$, define $\rho: L_{g,n} \to \Delta_{g,n}^{\mathrm{irr}}$ to be zero on this generator. If
$u=w$, let $S$ be the set of half-edges incident to
$w$, and let
$\alpha\in H^2(\Mb_{1,S})$ be its decoration. Define
$$
\rho(\Gamma,w,\alpha)
=
\lambda_S(\alpha)
(\Gamma,w,\deltairr).
$$

Since the functionals $\lambda_S$ are invariant under bijections of $S$, this defines a degree-zero map on graph coinvariants.
By construction,
$
\rho\circ\iota
=
\operatorname{id}_{\Delta_{g,n}^{\mathrm{irr}}}
$. 

It remains to prove
that $\rho$ commutes with $\deltasp$. This is immediate if the genus one vertex is nonspecial or if the splitting occurs at a nonspecial vertex.
We therefore consider a 
splitting of the special genus one vertex.
Write
$
S=U\sqcup T,
$
where the new genus zero and genus one vertices receive the half-edges in $U$ and $T$, respectively, and let $p$ and $q$
be the corresponding half-edges of the new edge. Stability gives that $|U| \geq 2$, while bridgelessness forces $T \neq \emptyset$. 

In the degree two K\"unneth decomposition of the clutching pullback, $\rho$ kills the component on the genus zero factor. After fixing that factor, the component on the genus one factor is $\theta^*_{U,T}(\alpha)$, where
$$
\theta_{U,T}\colon
\Mb_{1,T\cup\{q\}}
\longrightarrow
\Mb_{1,S}
$$
is the gluing map obtained by fixing a curve in $\Mb_{0,U \cup \{p\} }$ as in \eqref{eq: theta gluing map}. It is precisely the component of the K\"unneth
decomposition associated to the genus one vertex. Thus it suffices to prove
\begin{equation}\label{eq:lambda-compatible-with-splitting}
\lambda_{T\cup\{q\}}
\circ \theta_{U,T}^*
=
\lambda_S.
\end{equation}

Set $R=T\cup\{q\}$. Since $T\neq \emptyset$, the genus one relation gives
$$
12\lambda_R(\psi_q)
=
\lambda_R(\deltairr)
+
12\lambda_R(\delta_{0,R})=1-1=0.
$$
The pullback formulas \eqref{eq: theta^* for k_1, deltairr, psi} and \eqref{eq: theta^* for boundary A neq S} now give
$$
\lambda_R\bigl(\theta_{U,T}^*(\deltairr)\bigr)=1
$$
and
$$
\lambda_R\bigl(\theta_{U,T}^*(\delta_{0,B})\bigr)
=
\begin{cases}
-\dfrac{1}{12},&B=S,\\
~~~~0,&B\subsetneq S.
\end{cases}
$$
Indeed, the exceptional pullback for $B=U$ is $-\psi_q$, the
pullback for $B=S$ is the full boundary class $\delta_{0,R}$, and
every other nonzero pullback is a proper boundary class. This proves
\eqref{eq:lambda-compatible-with-splitting}.

Finally, $\rho$ changes neither the underlying graph nor its
structural-edge ordering, so corresponding splitting terms have the
same signs. Since it also preserves whether the output is bridgeless,
we conclude that
\begin{equation}\label{eq:rho com w/ dsp}
\rho\circ\deltasp
=
\deltasp\circ\rho.
\end{equation}
Thus $\rho$ is the required retraction of complexes.

\end{proof}

Theorem~\ref{thm:intro-weight two-model} follows by combining the
following theorem with
Proposition~\ref{prop: Delta_irr is weight zero}.

\begin{theorem}\label{thm: Split GK into delta irr and bar GK}
For every $g,n\geq0$ satisfying $2g-2+n>0$, there is a noncanonical
zigzag of quasi-isomorphisms of dg vector spaces
$$
\GKtwo
\simeq
\Delta_{g,n}^{\mathrm{irr}}
\oplus
\GKbartwo.
$$
\end{theorem}

\begin{proof}
Let
$
q\colon\GKtwo\longrightarrow\GKbartwo
$
be the quotient map. Following the proof of
\cite[Theorem~3.3]{PayneWillwacherWeight2}, filter both complexes by
the number of vertices. For $p\geq0$, let
$\mathcal F^p\GKtwo$ and $\mathcal F^p\GKbartwo$ be the subcomplexes
spanned by graphs with at least $p$ vertices. These define decreasing
filtrations, and $q$ is a filtered map. 

The differential $\deltal$ preserves the number of vertices, whereas
$\deltasp$ increases it by one. Hence
$$
\operatorname{gr}_{\mathcal F}\GKtwo
\cong
(\GKtwo,\deltal),
\qquad
\operatorname{gr}_{\mathcal F}\GKbartwo
\cong
(\GKbartwo,0),
$$
and the differential $d_1$ on each $E_1$-page is induced by
$\deltasp$.

We first identify the $E_1$-page of $\GKtwo$. Using the pointed loopless-core decomposition
\eqref{eq: pointed lc decomp}, the single-vertex decomposition
\eqref{eq: C=VtimesV}, the K\"unneth formula, and exactness
of finite-group coinvariants over $\QQ$, we obtain
$$
H\bigl(C_{(\Gamma,w)},\deltal\bigr)
\cong
\bigg(
H(V^2_{g(w),n(w)})
\otimes
\bigotimes_{\substack{v\in V(\Gamma)\\v\neq w}}
H(V^0_{g(v),n(v)})
\otimes
\QQ[-1]^{\otimes|E(\Gamma)|}
\bigg)_{\operatorname{Aut}(\Gamma,w)}.
$$
Equation~\eqref{eq:weight zero-single-vertex-cohomology} eliminates
every core having a nonspecial vertex of positive genus, while
Proposition~\ref{prop: coh of single vertex complex weight 2}
eliminates every core whose special vertex has genus at least two.
The remaining classes have representatives of the following three types:
\begin{enumerate}
\item loopless graphs whose vertices all have genus zero, with one
special vertex carrying an $H^2$-decoration;

\item loopless graphs whose special vertex has genus one and is
decorated by $\delta_{irr}$, with all remaining vertices
of genus zero;

\item graphs whose vertices all have genus zero with exactly one loop incident to the special
vertex, whose decoration represents a class in
$
H^2(\Mb_{0,n(w)+2})_{\mathbb S_2}
\big/
\xi^*H^2(\Mb_{1,n(w)}).
$
\end{enumerate}

The quotient map induces
$$
q_1\colon H(\GKtwo,\deltal)\longrightarrow\GKbartwo
$$
on the $E_1$-pages.  By the definition of $\GKbartwo$ and
Proposition~\ref{prop: coh of single vertex complex weight 2},
$q_1$ identifies the classes of types~\textup{(1)} and~\textup{(3)}
with $\GKbartwo$. It kills the classes of type~\textup{(2)}, since
these contain a positive-genus vertex. Such classes are precisely the
generators of $\Delta_{g,n}^{\mathrm{irr}}$. Therefore $q_1$ is
surjective and
$$
\ker(q_1)
=
\Delta_{g,n}^{\mathrm{irr}}.
$$

For a graph with $p$ vertices and $e$ structural edges, the total
cohomological degree is $e+2$, so its row in the spectral sequence is
\begin{equation}\label{eq:row formula}
s=e+2-p=g+1-\sum_v g(v).
\end{equation}
Thus type~\textup{(2)} lies in row $s=g$, while
types~\textup{(1)} and~\textup{(3)} lie in row $s=g+1$.
Since $\deltasp$ preserves the sum of the vertex genera, the
identifications above give
$$
\bigl(E_1(\GKtwo),d_1\bigr)
\cong
(\Delta_{g,n}^{\mathrm{irr}},\deltasp)
\oplus
(\GKbartwo,\deltasp).
$$
The only possible higher differential is
$$
d_2\colon E_2^{p,g+1}\longrightarrow E_2^{p+2,g}.
$$
We show that it vanishes by lifting each class in the upper row to
a cocycle for the full differential.

Let $[x]\in E_2^{p,g+1}$. Choose a $\deltal$-cycle representative
$x$ consisting of graphs with $p$ vertices, all of genus zero,
and at most one structural loop. Since $[x]$ is $d_1$-closed,
there is a chain $y$ with $p+1$ vertices in row $s=g$ such that
\begin{equation}\label{eq:d2-correction-equation}
\deltasp x=-\deltal y.
\end{equation}
Decompose $y$ by its number of structural loops:
$$
y=\sum_{\ell\geq0}y^{(\ell)}.
$$
Vertex splitting does not increase this number, while loop addition
increases it by one. Since $\deltasp x$ has at most one loop,
comparing loop counts in \eqref{eq:d2-correction-equation} gives
$$
\deltal y^{(\ell)}=0\qquad(\ell\geq1).
$$
We may therefore replace $y$ by $y^{(0)}$ without changing
\eqref{eq:d2-correction-equation}, and assume that $y$ is loopless.

By \eqref{eq:row formula}, every graph in $y$ has exactly one
genus one vertex and all remaining vertices of genus zero.
Thus $y\in L_{g,n}$, and $\deltasp y\in L_{g,n}$ as well.
Moreover,
$$
\deltal\deltasp y
=-\deltasp\deltal y
=\deltasp^2x
=0.
$$
Lemma~\ref{lem:loopless-genus one-kernel} gives
$\deltasp y\in\Delta_{g,n}^{\mathrm{irr}}$.

Let $\rho$ be the retraction from
Lemma~\ref{lem:corrected-rho}, and set
$$
y'=y-\rho(y).
$$
Since $\deltal$ vanishes on $\Delta_{g,n}^{\mathrm{irr}}$
and $\rho$ commutes with $\deltasp$ \eqref{eq:rho com w/ dsp}, we have
$$
\deltal y'=-\deltasp x,
\qquad
\deltasp y'
=\deltasp y-\rho(\deltasp y)=0.
$$
Consequently,
$$
(\deltal+\deltasp)(x+y')=0.
$$
The correction $y'$ has $p+1$ vertices, so this cocycle lifts the
original class $[x]$. In particular, $d_2[x]=0$.

All higher differentials vanish for row-degree reasons, so
$E_2=E_\infty$. Since the filtration is finite, the associated graded
of $H^\bullet(\GKtwo)$ is
$$
H^\bullet(\Delta_{g,n}^{\mathrm{irr}})
\oplus H^\bullet(\GKbartwo).
$$
Splitting this filtration degree by degree over $\QQ$ gives a
noncanonical isomorphism
$$
H^\bullet(\GKtwo)
\cong H^\bullet(\Delta_{g,n}^{\mathrm{irr}})
\oplus H^\bullet(\GKbartwo).
$$
Finally, choosing cocycle representatives gives a quasi-isomorphism
from the cohomology of a complex, equipped with zero differential,
to the complex itself. 
\end{proof}

\subsection{A weight zero identification}

The complex $\Delta_{g,n}^{\mathrm{irr}}$ has a direct interpretation
in terms of weight zero. Replace its special genus one vertex, decorated by $\delta_{irr}$, with a genus zero vertex carrying an additional marked leg. This operation lowers the genus by one and intertwines the differentials, giving the following identification. 

\begin{proposition}\label{prop: Delta_irr is weight zero}
For $g\geq1$ and $2g+n>3$, there is a natural isomorphism of
complexes
$$
\Delta_{g,n}^{\mathrm{irr}}
\cong
\overline{\mathrm{GK}}_{g-1,n+1}^{{bl},0}[-2].
$$
\end{proposition}

\begin{proof}
Let $w$ be the special genus one vertex of a generator of
$\Delta_{g,n}^{\mathrm{irr}}$. Define
$$
\mu_{g,n}\colon
\Delta_{g,n}^{\mathrm{irr}}
\longrightarrow
\overline{\mathrm{GK}}_{g-1,n+1}^{{bl},0}[-2]
$$
by changing the genus of $w$ from one to zero, removing its
$\delta_{{irr}}$-decoration, and attaching a new leg
labeled $n+1$ to $w$. The underlying graph and its edge ordering are left
unchanged.

The resulting graph has genus $g-1$ and remains bridgeless and
loopless. Only stability at $w$ requires verification. If $w$ had valence one, its sole incident half-edge could not belong to a structural edge, since that edge would be a bridge. It must therefore be a marked leg, and connectedness forces the graph to be the graph with one vertex and one marked leg, i.e., $(g,n)=(1,1)$. This case is excluded by the assumption $2g+n >3$. Thus $w$ has valence at least two, and the additional leg makes its image stable as a genus zero vertex.

Conversely, start with a generator of
$\overline{\mathrm{GK}}_{g-1,n+1}^{{bl},0}$. Remove the leg
labeled $n+1$, change the genus of its incident vertex from zero to
one, declare this vertex special, and decorate it by
$\delta_{{irr}}$. These constructions are inverse to one
another. Moreover, the labeled leg $n+1$ records the special vertex, so they identify the corresponding automorphism groups and preserve the relations arising from reordering the structural edges.

Since $\deltal$ vanishes on both complexes, it remains to compare $\deltasp$.
Splittings at nonspecial
vertices agree immediately. When the special vertex is split, the surviving terms contain one genus one vertex and one genus zero vertex, with the $\delta_{irr}$-decoration remaining on the genus one vertex by
$$
\theta^*(\delta_{{irr}})
=
\delta_{{irr}}.
$$
Under
$\mu_{g,n}$, this vertex corresponds precisely to the genus zero vertex receiving the leg labeled $n+1$. If it receives no original half-edges, the new edge is a bridge on the left, while the corresponding vertex is unstable on the right. Otherwise, the stability and bridgelessness conditions agree. Thus $\mu_{g,n}$ identifies the retained splitting terms. Since the new structural edge is placed first on both sides, their signs also agree.

Finally, removing the $\delta_{{irr}}$-decoration lowers the
cohomological degree by two, accounting for the shift $[-2]$.
\end{proof}

\begin{remark}
The preceding proposition identifies the
$\Delta_{g,n}^{\mathrm{irr}}$-summand of
$\mathrm{GK}_{g,n}^{{bl},2}$ in terms of weight zero
cohomology. It is natural to ask whether the remaining reduced summand
$\overline{\mathrm{GK}}_{g,n}^{{bl},2}$ admits a similar
description.

For comparison, consider the ordinary weight two Getzler--Kapranov
complex $\mathrm{GK}_{g,n}^{2}$, in which bridges are allowed. When
$n=0$ and $g\geq2$, the cohomology of
$\mathrm{GK}_{g,0}^{2}$, equivalently that of its reduced quotient
$\overline{\mathrm{GK}}_{g,0}^{2}$, is expressed as a direct sum of
four terms constructed from weight zero compactly supported cohomology;
see \cite[Theorems~1.2 and~3.3]{PayneWillwacherWeight2}. For $n>0$,
several injections from terms built from weight zero cohomology into
$H^\bullet(\mathrm{GK}_{g,n}^{2})$ are constructed in
\cite[Theorems~1.6 and~1.7]{PayneWillwacherWeight2}. At present, we do
not have analogous descriptions or injections for
$H^\bullet(\overline{\mathrm{GK}}_{g,n}^{{bl},2})$.
\end{remark}

\section{Comparing $\mathcal M_g$ and $\mathcal A_g$ through compact type}\label{sec: Comparing Mg and Ag through compact type}

We now use $\mathcal M_g^{ct}$ as an intermediary between
$\mathcal M_g$ and $\mathcal A_g$. The open inclusion
and the extended Torelli morphism allow us to transfer cohomological information among these three spaces. We use this comparison to uncover additional structure
in the weight zero compactly supported cohomology of
$\mathcal M_g^{ct}$ and to relate structures on
$\mathcal M_g$ and $\mathcal A_g$ arising from a priori different
constructions.

\subsection{Comparison with $\mathcal M_g$}
To prove Theorem~\ref{Thm: intro M_g comparison}, we first establish
the following comparison theorem, which shows that adjoining the
separating boundary does not change compactly supported cohomology
below degree $2g$ and kills no classes in degree $2g$.
Together with the known vanishing of $H_c^k(\M_g)$ for $k<2g$,
this gives the same vanishing for $\M_g^{ct}$.
Nonvanishing in degree $2g$ then follows by Poincaré duality from
the existence of a nonzero tautological class in
$H^{4g-6}(\M_g^{ct})$; see
\cite{FaberPandharipandeRelativeMaps,GraberVakilRelativeLocalization}.

\begin{theorem}\label{thm: i: Mg to Mgct}
   For every $g \geq 2$, let $i: \M_g \hookrightarrow \M_g^{ct}$ denote the open inclusion. The natural map
    \[
    i_!: H_c^k(\M_g) \to H_c^k(\M_g^{ct})
    \]
    is an isomorphism for $k<2g$ and is injective for $k=2g$.
\end{theorem}
\begin{proof}
    Set
$D_g:=\mathcal M_g^{ct}\setminus\mathcal M_g$. The localization sequence
$$
\cdots \to H_c^{k-1}(D_g) \to H_c^k(\M_g) \xrightarrow[]{i_!} H_c^k(\M_g^{ct}) \to H_c^k(D_g) \to \cdots
$$
shows it suffices to prove 
$
H_c^k(D_g)=0
$
for every $k<2g$.

The strata of $D_g$ are indexed by stable trees $T$ with at least one edge. The stratum indexed by $T$ is
$$
\mathcal S_T
\cong
\left[
\prod_{v\in V(T)}
\mathcal M_{g(v),n(v)}
\big/
\operatorname{Aut}(T)
\right].
$$
Here $g(v)$ is the genus of the vertex $v$, and $n(v)$ is its
valence, counting incident half-edges. Harer's vcd calculation \cite{Harer1986}, together with
Poincaré duality, gives
$$
H_c^k(\mathcal M_{0,n})=0
\qquad (k<n-3).
$$
For $h>0$,
\cite[Proposition~2.1]{BergstromFaberPayne2024} gives
$$
H_c^k(\mathcal M_{h,n})=0
\quad\text{for}\quad
k<
\begin{cases}
2h, & n=1,\\
2h-2+n, & n\geq2.
\end{cases}
$$

Write $\operatorname{mindeg}H_c^*(X)
:=
\inf\{k:H_c^k(X)\neq0\}.$ Let $L(T)$ be the number of leaves of $T$, i.e., the number
of vertices of valence one. Stability forces every leaf to have positive genus. The K\"unneth formula and the exactness of finite group invariants over $\QQ$ therefore give
\begin{align*}
\operatorname{mindeg}H_c^*(\mathcal S_T)
&\geq  \sum_{\substack{g(v)>0\\ n(v)=1}}2g(v)+ \sum_{\substack{g(v)>0\\ n(v)\geq2}} (2g(v)-2+n(v))+\sum_{g(v)=0} (n(v)-3)\\
& = \sum_{\substack{g(v)>0\\ n(v)=1}}(2g(v)-1)+L(T)+ \sum_{\substack{g(v)>0\\ n(v)\geq 2}} (2g(v)-2+n(v))+\sum_{g(v)=0} (n(v)-3)       \\
&= \sum_{g(v)>0}\bigl(2g(v)-2+n(v)\bigr)
+
\sum_{g(v)=0}\bigl(n(v)-3\bigr)
+L(T).
\end{align*}

Let $V_0(T)$ denote the number of genus zero vertices. Counting the number of half-edges gives $$\sum_{v\in V(T)}n(v)=2|E(T)|.$$ Furthermore, since $T$
is a tree,
$$
\sum_{v\in V(T)}g(v)=g
\qquad \text{and}\qquad
|E(T)|=|V(T)|-1.
$$
Consequently, the preceding lower bound simplifies as follows:
\begin{align*}
\operatorname{mindeg}H_c^*(\mathcal S_T)
&\geq
\sum_{g(v)>0}\bigl(2g(v)-2+n(v)\bigr)
+\sum_{g(v)=0}\bigl(n(v)-3\bigr)+L(T)\\
&=
2\sum_v g(v)+\sum_v n(v)
-2|V(T)|-V_0(T)+L(T)\\
&=
2g+2|E(T)|-2|V(T)|-V_0(T)+L(T)\\
&=
2g-2-V_0(T)+L(T).
\end{align*}
Moreover,
$$
-2
=
\sum_{v\in V(T)}\bigl(n(v)-2\bigr)
=
-L(T)
+
\sum_{n(v)\geq3}\bigl(n(v)-2\bigr).
$$
Every genus zero vertex occurs in the final sum and contributes at least
one. Hence
$$
L(T)\geq V_0(T)+2,
$$
and therefore
$$
\operatorname{mindeg}H_c^*(\mathcal S_T)\geq2g.
$$

For $r\geq1$, let $Z_r\subseteq D_g$ be the closed locus of curves
with at least $r$ nodes. These loci give a finite filtration
$$
D_g=Z_1\supseteq Z_2\supseteq\cdots\supseteq Z_{N+1}=\varnothing,
$$
whose successive layers are
$$
Z_r\setminus Z_{r+1}
=
\bigsqcup_{|E(T)|=r}\mathcal S_T.
$$
The associated compactly supported cohomology spectral sequence is
$$
E_1^{-r,q}
=
\bigoplus_{|E(T)|=r}H_c^{q-r}(\mathcal S_T)
\Longrightarrow H_c^{q-r}(D_g);
$$
see \cite[(3)]{PetersenSpectralSequence}.
The preceding estimate shows that the $E_1$-page vanishes in total
degrees below $2g$. Since the filtration is finite, all
associated-graded pieces of $H_c^k(D_g)$ vanish for $k<2g$,
and hence
$$
H_c^k(D_g)=0\qquad(k<2g).
$$
\end{proof}

In degree $2g$, this theorem restricts to an injection
$
W_0H_c^{2g}(\M_g)
\hookrightarrow
W_0H_c^{2g}(\M_g^{ct}).
$ We recall the classes to which we will apply this injection. For every odd
$g\geq3$, the wheel graph $W_g$ determines a nonzero class
$$
[W_g]
\in
\left(
W_0H_c^{2g}(\mathcal M_g)
\right)^\vee.
$$
For $k\geq1$, choose
$
\sigma_{2k+1}
\in W_0H_c^{4k+2}(\mathcal M_{2k+1})
$
to be the classes corresponding to the free generators in
\cite[Theorem~2.6]{ChanGalatiusPayne2021}.
These classes are detected by the corresponding wheel graphs. The resulting map
$$
\mathbb L(\sigma_3,\sigma_5,\sigma_7,\ldots)
\lhook\joinrel\longrightarrow
\bigoplus_{g\geq2}
W_0H_c^{2g}(\mathcal M_g)
$$
from the free genus-graded Lie algebra on these classes is injective by
\cite[(2.4.2)]{ChanGalatiusPayne2021}. Theorem~\ref{thm: i: Mg to Mgct}
immediately gives the following corollary.

\begin{corollary}
The composite
$$
\mathbb L(\sigma_3,\sigma_5,\sigma_7,\ldots)
\lhook\joinrel\longrightarrow
\bigoplus_{g\geq2}
W_0H_c^{2g}(\mathcal M_g)
\xrightarrow{\ \bigoplus i_{!}\ }
\bigoplus_{g\geq2}
W_0H_c^{2g}(\mathcal M_g^{ct})
$$
is injective as a map of genus-graded vector spaces.
\end{corollary}

\begin{corollary}
The dimension
$
\dim_{\QQ}
W_0H_c^{2g}(\mathcal M_g^{ct})
$
grows at least exponentially with $g$.
\end{corollary}

\begin{proof}
The genus-graded dimensions of
$\mathbb L(\sigma_3,\sigma_5,\sigma_7,\ldots)$
grow at least exponentially; see
\cite[the proof of Theorem~2.7]{ChanGalatiusPayne2021}.
\end{proof}

We now uncover a different structure in the weight zero cohomology of
compact type. Specifically, monomials in the odd wheels determine a polynomial family
of classes in its dual.

\subsection{Polynomial classes generated by odd wheels}
\label{subsec: Polynomial classes generated by odd wheels}

We use the structure of the weight zero cohomology of
$\mathcal A=\bigsqcup_{g\geq0}\mathcal A_g$ established in
\cite{BrownChanGalatiusPayneHopf} to construct a polynomial family in
the dual of $W_0H_c^\bullet(\mathcal M_g^{ct})$. For a partition
$\lambda=(h_1,\ldots,h_r)$ of $g$ into odd integers $h_i\geq3$,
we consider graphs obtained by pinching together the wheels
$W_{h_1},\ldots,W_{h_r}$. We prove that these graphs are cycles, that
any two pinched wheel graphs of type $\lambda$ are homologous, and
that the resulting classes are linearly independent as $\lambda$
varies.

Dualizing the isomorphism of
Theorem~\ref{thm:intro-weight zero-model} gives
\begin{equation}\label{eq:dual-bridgeless-weight zero}
H_k\bigl((\overline{\mathrm{GK}}_{g}^{bl,0})^\vee\bigr)
\cong
\left(W_0H_c^k(\mathcal M_g^{ct})\right)^\vee.
\end{equation}
We identify $(\overline{\mathrm{GK}}_{g}^{bl,0})^\vee$ with the chain
complex spanned by the same oriented graph generators
$[\Gamma,\omega]$, now with the edge-contraction differential. If
$\omega=(e_1<\cdots<e_m)$, then
\begin{equation*}\label{eq:edge-contraction-differential}
\partial[\Gamma,\omega]
=
\sum_{i=1}^m
(-1)^{i-1}
[\Gamma/e_i,\omega\setminus e_i].
\end{equation*}

For every odd integer $h\geq3$, let $W_h$ denote the wheel graph
with $h$ spokes and $2h$ edges. Every edge of $W_h$ lies in a
triangle, so contracting any edge produces a pair of parallel edges.
The resulting graph vanishes by the odd automorphism interchanging
these edges. Hence
$$
\partial W_h=0.
$$
Thus $W_h$ is a cycle; see also
\cite[Example~2.5]{ChanGalatiusPayne2021}.

Every nonzero edge contraction of a bridgeless graph is again
bridgeless. Consequently, there is an inclusion of chain complexes
$$
(\overline{\mathrm{GK}}_{g}^{bl,0})^\vee
\lhook\joinrel\longrightarrow
(\overline{\mathrm{GK}}_{g}^{0})^\vee.
$$
The wheel class is nonzero in the ordinary edge-contraction graph
complex by \cite[Theorem~2.6]{ChanGalatiusPayne2021}, as a consequence
of \cite[$\mathsection$9]{Willwacher2015}. It is therefore also
nonzero in the bridgeless subcomplex
$$
0\neq[W_h]
\in
H_{2h}\bigl((\overline{\mathrm{GK}}_{h}^{bl,0})^\vee\bigr).
$$

We now extend this construction to pinched collections of odd wheels.

\begin{definition}\label{def:wheel-gluing-datum}
Let $\lambda=(h_1,\ldots,h_r)$ be a partition of $g$ into odd
integers $h_i\geq3$. For each $i$, let $\mathcal W_i$ be a copy
of $W_{h_i}$ indexed by $i$.

A \emph{wheel-gluing datum of type $\lambda$} is a pair
$(\mathcal T,\mathbf v)$, where
\begin{enumerate}
\item $\mathcal T$ is a tree with vertex set $\{1,\ldots,r\}$;

\item $\mathbf v=(v_1,\ldots,v_r)$ is a collection of maps
$$
v_i\colon E_{\mathcal T}(i)\longrightarrow V(\mathcal W_i),
$$
where $E_{\mathcal T}(i)$ denotes the set of edges of
$\mathcal T$ incident to $i$. The maps $v_i$ need not be
injective.
\end{enumerate}

Let $\sim$ be the equivalence relation on
$\bigsqcup_{i=1}^r V(\mathcal W_i)$ generated by
$$
v_i(e)\sim v_j(e)
\qquad
\text{for every }e=\{i,j\}\in E(\mathcal T).
$$
The quotient graph
$$
\Gamma_{\lambda,\mathcal T,\mathbf v}
:=
\left(\bigsqcup_{i=1}^r\mathcal W_i\right)\big/\sim
$$
is called the \emph{pinched wheel graph determined by}
$(\lambda,\mathcal T,\mathbf v)$.
\end{definition}

Fix an edge ordering on each wheel factor and order the edges of a
pinched wheel graph by concatenation. Since $W_{h_i}$ has $2h_i$
edges, permuting the wheel factors changes the edge ordering by an even
permutation. The resulting orientation is therefore independent of the
ordering of the parts of $\lambda$.

Since the attachment maps $v_i$ need not be injective, several wheel factors
adjacent to $\mathcal W_i$ may be pinched to the same vertex of
$\mathcal W_i$. The tree $\mathcal T$ is auxiliary; its edges record
the vertex identifications and do not become edges of
$\Gamma_{\lambda,\mathcal T,\mathbf v}$.

\begin{example}
Let
$
\lambda=(7,3,5,5,9),
$
and let $\mathcal T$ be the following tree on the corresponding wheel
factors
$$
\mathcal T=
\vcenter{\hbox{
\begin{tikzpicture}[
    wheel/.style={
        circle,
        draw,
        minimum size=11mm,
        inner sep=1pt
    }
]
    \node[wheel] (W7)  at (0,0)    {$\W_1$};
    \node[wheel] (W3)  at (0,1.6)  {$\W_2$};
    \node[wheel] (W5a) at (-1.8,0) {$\W_3$};
    \node[wheel] (W5b) at (1.8,0)  {${\W_4}$};
    \node[wheel] (W9)  at (0,-1.6) {$\W_5$};

    \draw[thick] (W7)--(W3);
    \draw[thick] (W7)--(W5a);
    \draw[thick] (W7)--(W5b);
    \draw[thick] (W7)--(W9);
\end{tikzpicture}
}}.
$$
One choice of attachment maps $\mathbf v$ produces the following
pinched wheel graph

$$
\Gamma_{\lambda,\mathcal T,\mathbf v}
=
\vcenter{\hbox{
\begin{tikzpicture}[line cap=round,line join=round]

  \coordinate (v) at (-1.6,0);

  \attachedwheel{seven}{7}{v}{0}{1.6}
  \attachedwheel{three}{3}{v}{120}{0.82}
  \attachedwheel{fivea}{5}{v}{230}{1.02}

  \coordinate (w) at (seven-4);
  \attachedwheel{nine}{9}{w}{25.714}{1.12}

  \coordinate (u) at (seven-3);
  \attachedwheel{fiveb}{5}{u}{334.286}{0.98}

  \fill (v) circle (2pt);
  \fill (w) circle (2pt);
  \fill (u) circle (2pt);

  \node[font=\small] at (-2.55,1.72) {$W_3$};
  \node[font=\small] at (-2.55,-2.05) {$W_5$};
  \node[font=\small] at (0,-1.88) {$W_7$};
  \node[font=\small] at (2.85,2.47) {$W_9$};
  \node[font=\small] at (2.85,-2.30) {$W_5$};

\end{tikzpicture}
}}.
$$
\end{example}

\begin{remark}
The wheel-gluing datum is not determined by the resulting graph. In particular,
distinct trees and attachment maps may produce isomorphic quotients
$
\Gamma_{\lambda,\mathcal T,\mathbf v}
\cong
\Gamma_{\lambda,\mathcal T',\mathbf v'}.
$
Thus the gluing datum is a presentation of the pinched wheel graph, not
part of the graph itself. We will show that the resulting homology class
is independent of this presentation.
\end{remark}

For every partition $\lambda=(h_1,\ldots,h_r)$, let
$\Gamma_\lambda$ be the distinguished pinched wheel graph obtained by
identifying the central vertices of
$W_{h_1},\ldots,W_{h_r}$. We call $\Gamma_\lambda$ the
\emph{wheel bouquet of type $\lambda$}. We will show that every
pinched wheel graph of type $\lambda$ is homologous to
$\Gamma_\lambda$, and we denote their common class by
$
[W_\lambda].
$

\begin{proposition}\label{prop:pinched-wheel-independence}
Let
$\lambda=(h_1,\ldots,h_m)$
be a partition of $g$ into odd integers $h_i\geq3$. For every wheel-gluing datum
$(\mathcal T,\mathbf v)$ of type $\lambda$, the associated pinched
wheel graph
$\Gamma_{\lambda,\mathcal T,\mathbf v}$
is a cycle in
$(\overline{\mathrm{GK}}_{g}^{{bl},0})^\vee$.
Moreover, if
$(\mathcal T,\mathbf v)$ and
$(\mathcal T',\mathbf v')$
are any two wheel-gluing data of type $\lambda$, then
$$
[\Gamma_{\lambda,\mathcal T,\mathbf v}]
=
[\Gamma_{\lambda,\mathcal T',\mathbf v'}]
\in
H_{2g}
\bigl(
(\overline{\mathrm{GK}}_{g}^{{bl},0})^\vee
\bigr).
$$

In particular, every pinched wheel graph of type $\lambda$ is
homologous to the wheel bouquet $\Gamma_\lambda$. 

\end{proposition}

\begin{proof}

Every edge $e$ of
$\Gamma_{\lambda,\mathcal T,\mathbf v}$
belongs to one of its wheel factors and is contained there in a triangle.
After contracting $e$, the other two edges of this triangle become
parallel. The automorphism interchanging these edges reverses the edge
orientation, so the corresponding term vanishes in
$(\overline{\mathrm{GK}}_{g}^{{bl},0})^\vee$. Therefore
$
\partial\Gamma_{\lambda,\mathcal T,\mathbf v}=0,
$
and every pinched wheel graph is a cycle.

To prove the second part of the proposition, we first establish two local relations that allow us to move the
attachment point of an attached branch across a wheel.

\smallskip
\noindent
\emph{Sliding along a rim edge.}
Let $W_h$ be one of the wheel factors, let $u$ and $v$ be
adjacent rim vertices, and let $K$ be a connected branch of the
wheel tree attached to $W_h$ at $u$. Denote the resulting graph by
$\Gamma_u$, and let $\Gamma_v$ be the graph obtained by moving the
attachment of $K$ from $u$ to $v$, leaving all other attachments
unchanged.

Subdivide $e=uv$ by inserting a new vertex $x$, write
$e_1=ux$ and $e_2=xv$, and attach $K$ at the new vertex $x$. Denote this graph by $\widetilde{\Gamma}_{e,K}$. Orient it by replacing $e$ in the original
edge ordering with the consecutive edges $e_1<e_2$. Contracting
$e_1$ recovers $\Gamma_u$, whereas contracting $e_2$ recovers
$\Gamma_v$, and these two terms occur with opposite signs.

\begin{center}
\centering
\begin{tikzpicture}[
    scale=1.25,
    transform shape,
    font=\footnotesize,
    x=1cm,
    y=1cm,
    graph edge/.style={line width=0.75pt},
    distinguished edge/.style={line width=1.35pt},
    vertex/.style={circle,fill=black,inner sep=1.35pt},
    marked vertex/.style={
        circle,
        draw=black,
        fill=white,
        line width=0.85pt,
        inner sep=2pt
    },
    component/.style={
        circle,
        draw=black,
        fill=white,
        minimum size=9mm,
        inner sep=1pt,
        font=\small
    },
    attachment/.style={
        densely dashed,
        line width=0.65pt
    }
]

\newcommand{\Wfivebase}[1]{%
    \coordinate (#1-c) at (0,0);
    \coordinate (#1-u) at (-1.30,0.42);
    \coordinate (#1-v) at (0,1.36);
    \coordinate (#1-a) at (1.30,0.42);
    \coordinate (#1-b) at (0.80,-1.10);
    \coordinate (#1-d) at (-0.80,-1.10);

    \draw[graph edge]
        (#1-v)--(#1-a)--(#1-b)--(#1-d)--(#1-u);

    \foreach \p in {u,v,a,b,d}{
        \draw[graph edge] (#1-c)--(#1-\p);
    }

    \foreach \p in {c,u,v,a,b,d}{
        \node[vertex] at (#1-\p) {};
    }
}

\begin{scope}[xshift=0cm]
    \Wfivebase{L}

    \draw[distinguished edge]
        (L-u)--node[midway,above left=-4pt] {$e$} (L-v);

    \node[marked vertex] at (L-u) {};
    \node[marked vertex] at (L-v) {};
    \node[left=3pt] at (L-u) {$u$};
    \node[above=3pt] at (L-v) {$v$};

    \node[component] (L-K) at (-1.95,1.18) {$K$};
    \draw[attachment] (L-u)--(L-K.south east);

    \node at (0,-1.75) {$\Gamma_u$};
\end{scope}

\begin{scope}[xshift=4.20cm]
    \Wfivebase{M}
    \coordinate (M-x) at ($(M-u)!0.5!(M-v)$);

    \draw[distinguished edge] (M-u)--(M-x);
    \draw[distinguished edge] (M-x)--(M-v);

    \node[
        fill=white,
        inner sep=0.5pt,
        font=\footnotesize
    ] at (-1.08,0.9) {$e_1$};

    \node[
        fill=white,
        inner sep=0.5pt,
        font=\footnotesize
    ] at (-0.50,1.3) {$e_2$};

    \node[marked vertex] at (M-u) {};
    \node[marked vertex] at (M-v) {};
    \node[marked vertex] at (M-x) {};
    \node[left=3pt] at (M-u) {$u$};
    \node[above=3pt] at (M-v) {$v$};
    \node[below right=-2pt] at (M-x) {$x$};

    \node[component] (M-K) at (-1.25,1.70) {$K$};
    \draw[attachment] (M-x)--(M-K.south east);

    \node at (0,-1.75) {$\widetilde{\Gamma}_{e,K}$};
\end{scope}

\begin{scope}[xshift=8.40cm]
    \Wfivebase{R}

    \draw[distinguished edge]
        (R-u)--node[midway,above left=-4pt] {$e$} (R-v);

    \node[marked vertex] at (R-u) {};
    \node[marked vertex] at (R-v) {};
    \node[left=3pt] at (R-u) {$u$};
    \node[above=3pt] at (R-v) {$v$};

    \node[component] (R-K) at (-0.80,1.92) {$K$};
    \draw[attachment] (R-v)--(R-K.south east);

    \node at (0,-1.75) {$\Gamma_v$};
\end{scope}

\end{tikzpicture}
\end{center}

Every other contraction vanishes. Indeed, every edge outside $W_h$ remains in a triangle. Moreover, because
$e$ is a rim edge, every edge of $W_h$ other than $e$ lies in a
triangle not containing $e$. Contracting any such edge therefore
produces parallel edges. After changing the overall sign of
$\widetilde\Gamma_{e,K}$, if necessary, we obtain
$$
\partial\widetilde\Gamma_{e,K}
=
\Gamma_u-\Gamma_v.
$$
Consequently,
$$
[\Gamma_u]=[\Gamma_v].
$$
Since the rim vertices form a cycle, an attached branch may be moved
between any two rim vertices.

\smallskip
\noindent
\emph{Sliding from the rim to the central vertex.}
Let $c$ be the central vertex of $W_h$. If no branches are
attached at rim vertices, there is nothing to prove. Otherwise, using the
preceding rim-slide relation, move every branch attached at a rim
vertex of $W_h$ to one fixed rim vertex $r$. Regard the union of
these branches as a single rooted connected subgraph $K$. Branches already
attached at $c$ are left unchanged. Denote this graph by $\Gamma_r$, and let $\Gamma_c$ be
obtained by moving the attachment of $K$ from $r$ to $c$.

Subdivide the spoke $e=cr$ by inserting a new vertex $x$,
write $e_1=rx$ and $e_2=xc$, and attach $K$ at $x$. Let
$\widetilde\Gamma_{e,K}$ denote the resulting graph,
oriented by replacing $e$ in the original edge ordering
with the consecutive edges $e_1<e_2$. Contracting $e_1$ recovers $\Gamma_r$, whereas contracting
$e_2$ recovers $\Gamma_c$, and these terms occur with
opposite signs.

\begin{center}
\centering
\begin{tikzpicture}[
    scale=1.25,
    transform shape,
    font=\footnotesize,
    x=1cm,
    y=1cm,
    graph edge/.style={line width=0.75pt},
    distinguished edge/.style={line width=1.35pt},
    vertex/.style={circle,fill=black,inner sep=1.35pt},
    marked vertex/.style={
        circle,
        draw=black,
        fill=white,
        line width=0.85pt,
        inner sep=2pt
    },
    component/.style={
        circle,
        draw=black,
        fill=white,
        minimum size=9mm,
        inner sep=1pt,
        font=\small
    },
    attachment/.style={
        densely dashed,
        line width=0.65pt
    }
]

\newcommand{\Wfivebase}[1]{%
    \coordinate (#1-c) at (0,0);
    \coordinate (#1-r) at (0,1.36);
    \coordinate (#1-a) at (1.30,0.42);
    \coordinate (#1-b) at (0.80,-1.10);
    \coordinate (#1-d) at (-0.80,-1.10);
    \coordinate (#1-q) at (-1.30,0.42);

    \draw[graph edge]
        (#1-r)--(#1-a)--(#1-b)--(#1-d)--(#1-q)--cycle;

    \foreach \p in {a,b,d,q}{
        \draw[graph edge] (#1-c)--(#1-\p);
    }

    \foreach \p in {c,r,a,b,d,q}{
        \node[vertex] at (#1-\p) {};
    }
}

\begin{scope}[xshift=0cm]
    \Wfivebase{L}

    \draw[distinguished edge]
        (L-r)--node[midway,xshift=4pt] {$e$} (L-c);

    \node[marked vertex] at (L-r) {};
    \node[marked vertex] at (L-c) {};
    \node[above, yshift=.5pt] at (L-r) {$r$};
    \node[below,yshift=-2pt] at (L-c) {$c$};

    \node[component] (L-K) at (-1.10,1.70) {$K$};
    \draw[attachment] (L-r)--(L-K.south east);

    \node at (0,-1.75) {$\Gamma_r$};
\end{scope}

\begin{scope}[xshift=4.20cm]
    \Wfivebase{M}
    \coordinate (M-x) at ($(M-r)!0.5!(M-c)$);

    \draw[distinguished edge]
        (M-r)--node[midway,xshift=5.5pt] {$e_1$} (M-x);

    \draw[distinguished edge]
        (M-x)--node[midway,xshift=5.5pt] {$e_2$} (M-c);

    \node[marked vertex] at (M-r) {};
    \node[marked vertex] at (M-x) {};
    \node[marked vertex] at (M-c) {};
    \node[above, yshift=.5pt] at (M-r) {$r$};
    \node[left=3pt] at (M-x) {$x$};
    \node[below,yshift=-2pt] at (M-c) {$c$};

    \node[component] (M-K) at (-1.25,1.65) {$K$};
    \draw[attachment] (M-x)--(M-K.east);

    \node at (0,-1.75) {$\widetilde{\Gamma}_{e,K}$};
\end{scope}

\begin{scope}[xshift=8.40cm]
    \Wfivebase{R}

    \draw[distinguished edge]
        (R-r)--node[midway,xshift=4pt] {$e$} (R-c);

    \node[marked vertex] at (R-r) {};
    \node[marked vertex] at (R-c) {};
      \node[above, yshift=.5pt] at (R-r) {$r$};
   \node[below,yshift=-2pt] at (R-c) {$c$};

    \node[component] (R-K) at (-1.2,1.5) {$K$};
    \draw[attachment] (R-c)--(R-K.east);

    \node at (0,-1.75) {$\Gamma_c$};
\end{scope}

\end{tikzpicture}

\end{center}

All remaining contractions vanish by the parallel-edge relation,
except possibly those of the two rim edges adjacent to $r$. When
$h=3$, these contractions also produce parallel edges. Suppose
$h\geq5$. After contracting either adjacent rim edge, the resulting
graph admits a reflection fixing the central vertex $c$, the subdivision vertex
$x$, and the contracted rim vertex. This reflection extends by the
identity over $K$ and over every branch already attached at $c$.
On the edge set, it exchanges
$
({h-1})/{2}
$
pairs of rim edges and
$
({h-3})/{2}
$
pairs of spokes. It therefore acts by
$$
\frac{h-1}{2}+\frac{h-3}{2}=h-2
$$
transpositions. Since $h$ is odd, this permutation is odd, so the
corresponding oriented graph is zero.

Thus, after an overall choice of sign,
$
\partial\widetilde\Gamma_{e,K}
=
\Gamma_c-\Gamma_{r},
$
and hence
$
[\Gamma_c]=[\Gamma_{r}].
$
It follows that all branches incident to $W_h$ may be moved to its
central vertex.

Apply the two sliding relations successively to move all
attachments on each wheel factor of
$\Gamma_{\lambda,\mathcal T,\mathbf v}$ to its central vertex.
Each move preserves the homology class and changes attachment
vertices only on the wheel being processed. Thus, after all
wheel factors have been processed, every identification is
between central vertices. Since $\mathcal T$ is connected,
all central vertices are identified to a single vertex,
giving the wheel bouquet $\Gamma_\lambda$. Therefore
$$
[\Gamma_{\lambda,\mathcal T,\mathbf v}]
=
[\Gamma_\lambda],
$$
and the class depends only on the unordered partition $\lambda$.
\end{proof}

Recall that
$
\W:=\QQ[W_3,W_5,W_7,\ldots]
$
is graded by genus, with $W_h$ in genus $h$, and let $\W_g$
denote its genus $g$ summand. Proposition~\ref{prop:pinched-wheel-independence} defines a linear map
$$
\Phi_g\colon
\W_g
\longrightarrow
H_{2g}\!\left(
(\overline{\mathrm{GK}}_{g}^{{bl},0})^\vee
\right)
\cong
\left(
W_0H_c^{2g}(\mathcal M_g^{ct})
\right)^\vee,
\qquad
W_\lambda\longmapsto[W_\lambda].
$$

Together with the assignment $\Phi(1)=1$, these maps assemble into a genus graded linear map
$$
\Phi\colon
\W
\longrightarrow
\QQ
\oplus
\bigoplus_{g\geq2}
\left(
W_0H_c^{2g}
(\mathcal M_g^{ct})
\right)^\vee,
$$
where $\QQ$ is placed in genus zero. We now prove
Theorem~\ref{thm:intro-wheel-polynomial} by showing that $\Phi$ is
injective.

The proof uses the multiplicative structure on the Quillen $E^1$-page
recalled in Sections~\ref{sec:weight zero-Ag} and
\ref{sec:K-theory-of-graphs}. We first place the dual bridgeless graph
complex in this framework.

Recall from \eqref{eq:GrC-definition} that $C_\bullet$ is the complex of oriented finite
graphs, filtered by first Betti number, and
$$
\operatorname{Gr}_g(C)
=
F_gC_\bullet/F_{g-1}C_\bullet.
$$
Its differential \eqref{eq:GrC-differential} is
$$
\partial_{\operatorname{Gr}}
=
\partial_c^{\mathrm{nl}}
-
\partial_d^{\mathrm{br}},
$$
where the two terms contract non-loop edges and delete bridges,
respectively.

\begin{lemma}\label{lem:bridgeless-inclusion-into-GrC}
The identity on oriented graph generators defines an injective map of
complexes
$$
\jmath_g\colon
\GKbarzeroBM
\lhook\joinrel\longrightarrow
\operatorname{Gr}_g(C),
\qquad
[G,\omega]\longmapsto [G,\omega].
$$

\end{lemma}

\begin{proof}
The two complexes impose the same orientation relation. Every generator
of $\GKbarzeroBM$ is loopless and bridgeless, so $\partial_{\mathrm{Gr}}$ restricts to the edge contraction differential. If a contraction produces parallel edges, the resulting
graph vanishes in both complexes by the odd automorphism interchanging
them. Thus $\jmath_g$ is an injective map of complexes.
\end{proof}

Composing the induced map on homology with the map \eqref{eq:graph-to-Quillen} defines
$$
\Psi_g
:=
\rho_{g,g}\circ(\jmath_g)_*
\colon
H_{2g}\!\left(
(\overline{\mathrm{GK}}_g^{{bl},0})^\vee
\right)
\longrightarrow
{}^QE^1_{g,g}.
$$

\begin{proof}[Proof of Theorem~\ref{thm:intro-wheel-polynomial}]
The genus zero component of $\Phi$ is the identity on $\QQ$, so it suffices to prove injectivity in each positive genus.

We first compute $\Psi_h$ on a single odd wheel $W_h$, with $h \geq 3$.
Choose an edge
ordering $\omega_h$ compatible with the orientation convention of
\cite{BrownSchnetzWheelClasses}. By \eqref{eq:rho-graph-Laplacian},
$
\rho_{h,h}
=
\theta_{h,h}\circ(L_h)_*\circ\alpha_h
$. Here $L_h$ is the map from \eqref{eq:Lg map}.
First note
$$\alpha_h((\jmath_h)_*([W_h, \omega_h]))=[\sigma_{W_h},\omega_h]\in
\widetilde H_{2h}\!\left(
\operatorname{Gr}_h(M_{\mathrm{Gr}})
\right),$$ where the right hand side denotes the homology class represented by the oriented metric graph cell indexed by $W_h$. Choose the integral basis of $H_1(W_h;\ZZ)$ used in \cite[$\mathsection$1.1,~Figure~1]{BrownSchnetzWheelClasses}. In this basis, the quadratic form $Q_{W_h,\ell}$ is represented by the graph-Laplacian matrix $\Lambda_{W_h}(\ell)$ displayed in \cite[$\mathsection$1.1, equation~(4)]{BrownSchnetzWheelClasses}. The restriction of $L_h$ to the wheel
cell is therefore
$
\ell\longmapsto[\Lambda_{W_h}(\ell)].
$
The corresponding cone in $P_h$ is
$$
\tau_{W_h}
:=
\left\{
\Lambda_{W_h}(\ell)
\ \middle|\
\ell_e>0
\right\}.
$$
Let $[\tau_{W_h}]$ denote the image of the oriented wheel
cycle under $(L_h)_*$.
By \cite[$\mathsection$1.3 and Theorem~4(i)]
{BrownSchnetzWheelClasses},
$$
(L_h)_*[\sigma_{W_h},\omega_h]
=
[\tau_{W_h}]
\neq0.
$$

Consequently,
$
\Psi_h\bigl([W_h,\omega_h]\bigr)
=
\theta_{h,h}\bigl([\tau_{W_h}]\bigr)
\neq 0.
$
For simplicity we also denote 
$\theta_{h,h}([\tau_{W_h}])\in{}^QE^1_{h,h}$ by $[\tau_{W_h}]$.

Now let
$
\Gamma_\lambda
=
W_{h_1}\vee\cdots\vee W_{h_r}
$ 
be the wheel bouquet of type $\lambda$, with edge ordering
$\omega_\lambda=\omega_{h_1}*\cdots*\omega_{h_r}$. Multiplication in $\operatorname{Gr}(C)$ is induced by disjoint union \eqref{eq:GrC-product}.
We claim that the image of the wheel bouquet in
$H_\bullet(\operatorname{Gr}(C))$ agrees with the product of the
individual wheel classes. We compare the wedge and disjoint
union of two wheel bouquets $G$ and $H$. Let $B(G,H)$ be the graph
obtained by joining their central vertices by a new bridge $b$, ordered
before the edges of $G$ and $H$. Then
$$
\partial_{\operatorname{Gr}}
[B(G,H),b*\omega_G*\omega_H]
=
[G\vee H,\omega_G*\omega_H]
-
[G\sqcup H,\omega_G*\omega_H].
$$

Indeed, contracting $b$ gives $G\vee H$, while deleting $b$ gives
$G\sqcup H$.  Every remaining edge lies in a wheel triangle, so its
contraction produces parallel edges and hence vanishes.

Iterating this relation gives
$$
(\jmath_g)_*
[\Gamma_\lambda,\omega_\lambda]
=
\prod_{i=1}^r
(\jmath_{h_i})_*[W_{h_i},\omega_{h_i}]
$$
in $H_{2g}(\operatorname{Gr}_g(C))$.
Since $\rho$ is a morphism of bigraded
algebras, 
\begin{equation}\label{eq:canonical-bouquet-image}
\Psi_g\bigl([\Gamma_\lambda,\omega_\lambda]\bigr)
= \prod_{i=1}^r[\tau_{ W_{h_i}}]
\in{}^QE^1_{g,g}.
\end{equation}
No Koszul signs arise because each wheel has even homological degree.
It remains to prove linear independence. 

Suppose that, for some fixed
$g$, there is a relation
$$
\sum_{\lambda}q_\lambda
[\Gamma_\lambda,\omega_\lambda]
=
0
$$
in $H_{2g}(
(\overline{\mathrm{GK}}_{g}^{{bl},0})^\vee),
$
where $\lambda$ ranges over partitions of $g$ into odd integers at
least $3$. Applying $\Psi_g$ and using
\eqref{eq:canonical-bouquet-image} gives
$$
\sum_{\lambda}q_\lambda
\prod_{i=1}^r[\tau_{ W_{h_i}}]=0
$$
in ${}^QE^1_{g,g}$. By \cite[Theorem~7.4]{BrownChanGalatiusPayneHopf}, there is an
injective homomorphism of bigraded commutative algebras
$$
\mathbb Q[W_3,W_5,W_7,\ldots]
\otimes\mathbb Q[x]/(x^2)
\lhook\joinrel\longrightarrow{}^QE^1
$$
sending $W_h$ to $[\tau _{W_h}]$. The corresponding wheel monomials
are therefore linearly independent, so $q_\lambda=0$ for every
$\lambda$. Thus every $\Phi_g$ is injective, and hence so is
$\Phi$.
\end{proof}

\subsection{Injectivity under the extended Torelli morphism}\label{subsec: Injectivity under the extended Torelli morphism}

The preceding subsection constructed a polynomial family of classes in
the dual of $
W_0H_c^\bullet(\mathcal M_g^{ct})$. We now show that its image under the extended Torelli morphism remains linearly independent. Let
$$
\Phi \colon
\mathcal W
\lhook\joinrel\longrightarrow \QQ \oplus
\bigoplus_{g\geq2}
\left(
W_0H_c^{2g}
(\mathcal M_g^{ct})
\right)^\vee
$$
be the injection constructed in
Theorem~\ref{thm:intro-wheel-polynomial}. 
For $g\geq2$, write
$$
(\tau_g^{ct})_*
:=
\bigl((\tau_g^{ct})^*\bigr)^\vee.
$$
We identify the formal genus zero summand $\QQ$ with
$\left(W_0H_c^0(\A_0)\right)^\vee$.

\begin{theorem}
\label{thm: wheel classes inject under torelli}
The composite
$$
\W
\xrightarrow{\ \Phi\ }
\QQ\oplus
\bigoplus_{g\geq2}
\left(
W_0H_c^{2g}(\M_g^{ct})
\right)^\vee
\xrightarrow{\ \operatorname{id}_{\QQ}\oplus
\bigoplus_{g\geq2}(\tau_g^{ct})_*\ }
\bigoplus_{g\geq0}
\left(
W_0H_c^{2g}(\A_g)
\right)^\vee
$$
is injective. 
\end{theorem}

The proof rests on two compatibility statements relating the extended
Torelli pushforward to the map in \eqref{eq:graph-to-Quillen}.

\begin{lemma}\label{Lemma: Algebraic and tropical torelli comparison}
Let $g\geq2$. For every $k$, the following diagram commutes
$$
\begin{tikzcd}[column sep=large,row sep=large]
\left(W_0H_c^k(\M_g^{ct})\right)^\vee
    \arrow[r, "(\tau_g^{ct})_*"]
    \arrow[d, "\kappa_{\M^{ct}}^\vee"']
&
\left(W_0H_c^k(\A_g)\right)^\vee
    \arrow[d, "\kappa_{\A}^\vee"]
\\
H_k^{\mathrm{BM}}(M_g^{ct,\mathrm{trop}})
    \arrow[r, "(\tau_g^{ct,\mathrm{trop}})_*"']
&
H_k^{\mathrm{BM}}(A_g^{\mathrm{trop}}).
\end{tikzcd}
$$
Here the vertical maps are the canonical comparison isomorphisms, and
the lower map is Borel--Moore pushforward under the proper compact type
tropical Torelli map.
\end{lemma}

\begin{proof}
The local toroidal description of the extended Torelli morphism in
\cite[Theorem~3.2]{AlexeevBrunyate} induces the tropical Torelli map
$$
\tau_g^{ct,\mathrm{trop}}\colon
M_g^{ct,\mathrm{trop}}
\longrightarrow
A_g^{\mathrm{trop}}.
$$
On the cone indexed by a weighted stable graph $(\Gamma,w)$, the map
is given by
$$
(\ell_e)_{e\in E(\Gamma)}
\longmapsto
[Q_{\Gamma,\ell}].
$$
Here
$
|w|:=\sum_{v\in V(\Gamma)}w(v),
$
and $Q_{\Gamma,\ell}$ is regarded as a positive-semidefinite quadratic
form on
$
H_1(\Gamma;\RR)\oplus\RR^{|w|}
$
by extending it by zero on the second summand.
This is the tropical Jacobian construction of
\cite[Definition~4.1.1 and Theorem~4.1.7]
{VivianiTropicalizing}.

After passing to compatible smooth toroidal subdivisions, the
weight zero comparison isomorphisms
\cite[Corollary~5.6 and Theorem~5.8]{ChanGalatiusPayne2021}
and
\cite[Proposition~2.1]{BrownChanGalatiusPayneHopf}
identify the algebraic and tropical pullbacks with the maps induced by
this same map of cones. By functoriality of the weight zero comparison for proper
morphisms of toroidal pairs, we obtain
$$
\kappa_{\M^{ct}}
\circ
(\tau_g^{ct,\mathrm{trop}})^*
=
(\tau_g^{ct})^*
\circ
\kappa_{\A}.
$$
Dualizing gives the claimed square.
\end{proof}

\begin{lemma}
\label{lem:bridgeless-graph-complex-to-Quillen}
Fix $g\geq2$ and $t\geq0$, and set $k=g+t$. Then the following
diagram commutes
$$
\begin{tikzcd}[column sep=large,row sep=large]
H_k(\GKbarzeroBM)
    \arrow[r, "(\jmath_g)_*"]
    \arrow[d, "\nu_g"']
&
H_k\!\left(\operatorname{Gr}_g(C)\right)
    \arrow[d, "\rho_{g,t}"]
\\
H_k^{\mathrm{BM}}
\left(M_g^{ct,\mathrm{trop}}\right)
    \arrow[r,
    "(\iota_{\A})_{g,t}\circ
    (\tau_g^{ct,\mathrm{trop}})_*"']
&
{}^QE^1_{g,t}.
\end{tikzcd}
$$
Here $\nu_g$ is induced by the cellular assignment
$
[\Gamma,\omega]\longmapsto[\sigma_\Gamma,\omega],
$
where $\sigma_\Gamma$ is the open metric-graph cell
indexed by $\Gamma$, oriented by $\omega$.
\end{lemma}

\begin{proof}
Let
$
z=\sum_i a_i[G_i,\omega_i]
$
be a cycle in $\GKbarzeroBM$. First note that $\nu_g([z])$ and
$\alpha_g((\jmath_g)_*[z])$ are represented by the same cellular chain
$$
\sum_i a_i[\sigma_{G_i},\omega_i]
$$
of oriented edge-length cones, viewed respectively in
$M_g^{ct,\mathrm{trop}}$ and
$\operatorname{Gr}_g(M_{\mathrm{Gr}})$. 

By \eqref{eq:iota-A-concrete} and
\eqref{eq:rho-graph-Laplacian},
$$
(\iota_{\A})_{g,t}
=
\theta_{g,t}\circ\pi_g,
\qquad
\rho_{g,t}
=
\theta_{g,t}\circ(L_g)_*\circ\alpha_g.
$$
It therefore suffices to compare the two maps before applying
$\theta_{g,t}$.

Each $G_i$ has first Betti number $g$ and all vertex weights are zero. Hence, the compact type tropical Torelli map sends its
edge-length cone into the open rank-$g$ stratum
$$
P_g/\operatorname{GL}_g(\ZZ)
\subseteq
A_g^{\mathrm{trop}}
$$
by
$$
\ell
\longmapsto
\Bigg[
\sum_{e\in E(G_i)}
\ell_e(e^\vee)^2
\Bigg].
$$
This is precisely the restriction of $L_g$ to the cell indexed by
$G_i$. The two cellular models use the same edge-length coordinates
and the orientation determined by $\omega_i$, so the induced maps
agree on each cellular generator. Hence
$$
\pi_g\circ(\tau_g^{ct,\mathrm{trop}})_*\circ\nu_g
=
(L_g)_*\circ\alpha_g\circ(\jmath_g)_*.
$$
\end{proof}

\begin{proof}[Proof of Theorem~\ref{thm: wheel classes inject under torelli}]
All maps preserve the genus grading, and the genus zero component is the
identity on $\QQ$. It therefore suffices to prove injectivity in each
genus $g\geq2$.

Let
$$
z
=
\sum_{\lambda}q_\lambda W_\lambda
\in
\mathcal W_{g},
$$
where $\lambda=(h_1,\ldots,h_r)$ ranges over partitions of $g$ into
odd integers at least $3$. Set
$$
x_z
:=
\sum_{\lambda}q_\lambda
[\Gamma_\lambda,\omega_\lambda]
\in
H_{2g}(\GKbarzeroBM).
$$

By the construction of $\Phi_g$ and $\nu_g$,
$$
\kappa_{\M^{ct}}^\vee\bigl(\Phi_g(z)\bigr)
=
\nu_g(x_z).
$$
Suppose that
$
(\tau_g^{ct})_*\bigl(\Phi_g(z)\bigr)=0.
$ Lemma~\ref{Lemma: Algebraic and tropical torelli comparison}
and Lemma~\ref{lem:bridgeless-graph-complex-to-Quillen} give
$$
\begin{aligned}
0
&=
(\iota_{\A})_{g,g}
\circ\kappa_{\A}^\vee
\circ(\tau_g^{ct})_*
\bigl(\Phi_g(z)\bigr)
\\
&=
(\iota_{\A})_{g,g}
\circ(\tau_g^{ct,\mathrm{trop}})_*
\circ\nu_g(x_z)
\\
&=
\rho_{g,g}\bigl((\jmath_g)_*x_z\bigr)
\\
&=
\sum_{\lambda}q_\lambda
\prod_{i=1}^r[\tau_{ W_{h_i}}]
\qquad\text{in }{}^QE^1_{g,g},
\end{aligned}
$$
where the final equality is
\eqref{eq:canonical-bouquet-image}.

By \cite[Theorem~7.4]{BrownChanGalatiusPayneHopf}, the assignment
$
W_h\longmapsto[\tau _{W_h}]
$
extends to an injection
$$
\QQ[W_3,W_5,W_7,\ldots]
\lhook\joinrel\longrightarrow
{}^QE^1.
$$
The monomials
$
\prod_{i=1}^r[\tau_{W_{h_i}}]
$
are therefore linearly independent. Hence $q_\lambda=0$ for every
$\lambda$, and thus $z=0$.
It follows that
$$
(\tau_g^{ct})_*\circ\Phi_g
$$
is injective for every $g \geq 2$.
\end{proof}

\subsection{The motivic Lie algebra and the extended Torelli map}

We now use the compact type comparison to relate the Lie algebra
structures arising from $\mathcal M_g$ and $\mathcal A_g$.
Throughout this subsection, we work with real coefficients. Set
\[
D_{\mathcal A}
:=
\bigoplus_{g\geq0}
W_0H_c^{2g}(\mathcal A_g;\mathbb R),
\qquad
D_{\mathcal M}
:=
\bigoplus_{g\geq2}
W_0H_c^{2g}(\mathcal M_g;\mathbb R).
\]
The Hopf algebra structure constructed in
\cite{BrownChanGalatiusPayneHopf} restricts to
$D_{\mathcal A}$. Its primitive subspace is a genus-graded
Lie algebra under the commutator bracket.

Let $\mathsf{GC}_2^{(g)}$ denote the genus $g$ part of
Kontsevich's cohomological even commutative graph complex.
We use the genus-graded version
\[
H^0(\mathsf{GC}_2)
:=
\bigoplus_{g\geq2}H^0(\mathsf{GC}_2^{(g)}).
\]
The graph-complex description of weight zero cohomology gives
a genus-graded isomorphism
\begin{equation*}\label{eq:CGP-diagonal-identification}
\gamma_{\mathcal M}\colon
H^0(\mathsf{GC}_2)\otimes_{\mathbb Q}\mathbb R
\xrightarrow{\sim}
D_{\mathcal M};
\end{equation*}
see \cite[Theorems~1.2 and~1.3]{ChanGalatiusPayne2021}.
We use this isomorphism to transport the graph-insertion
bracket to $D_{\mathcal M}$.

The construction in
\cite[Proposition~1.8 and Section~1.5]
{BrownChanGalatiusPayneHopf}
gives a genus-graded morphism of Lie algebras
\[
\beta_{\mathcal A}\colon
\operatorname{Prim}(D_{\mathcal A})
\longrightarrow
H^0(\mathsf{GC}_2)\otimes_{\mathbb Q}\mathbb R.
\]

This is the map recorded in
\cite[equation~(29)]{BrownPanzerPortner2026}.
We set
\[
\eta_{\mathcal A}
:=
\gamma_{\mathcal M}\circ\beta_{\mathcal A}
\colon
\operatorname{Prim}(D_{\mathcal A})
\longrightarrow
D_{\mathcal M}.
\]

The next theorem identifies $\eta_{\mathcal A}$
geometrically. Pulling a primitive class back along the
extended Torelli morphism gives the same class as first
applying $\eta_{\mathcal A}$ and then extending by zero
from $\mathcal M_g$ to $\mathcal M_g^{ct}$.

\begin{theorem}
\label{thm:compact type-torelli-comparison}
For $g\geq2$, let
\[
P_g^{\mathcal A}
:=
\operatorname{Prim}(D_{\mathcal A})
\cap
W_0H_c^{2g}(\mathcal A_g;\mathbb R),
\]
and write $\eta_{\mathcal A,g}$ for the restriction of
$\eta_{\mathcal A}$ to $P_g^{\mathcal A}$.
Then, for every $\alpha\in P_g^{\mathcal A}$,
\[
(\tau_g^{ct})^*(\alpha)
=
i_{g,!}\bigl(\eta_{\mathcal A,g}(\alpha)\bigr).
\]
\end{theorem}

\begin{proof}
Fix $g\geq2$. All spaces and complexes in the proof are
understood over $\mathbb R$. We work with the dual maps because the compatibility
statements established in the preceding subsection are
expressed on graph homology. Set
\[
T_g:=(\tau_g^{ct})^*|_{P_g^{\mathcal A}}.
\]
Our goal is to prove
\(
T_g^\vee
=
\eta_{\mathcal A,g}^\vee\circ(i_{g,!})^\vee.
\)
We first express $T_g^\vee$ using the graph model.
Let
\[
\Theta_g
:=
\left(\kappa_{\mathcal M^{ct}}^\vee\right)^{-1}
\circ\nu_g
\colon
H_{2g}(\GKbarzeroBM)
\xrightarrow{\sim}
\left(
W_0H_c^{2g}(\mathcal M_g^{ct};\mathbb R)
\right)^\vee
\]
be the canonical graph comparison.

To account for the restriction to primitive classes,
let $R_g$ denote the composite
\[
{}^QE^1_{g,g}
\longrightarrow
\mathcal A^{\mathrm{BM}}_{g,g}
\longrightarrow
(P_g^{\mathcal A})^\vee.
\]
The first arrow is induced by the canonical Hopf quotient
${}^QE^1/(e)\cong\mathcal A^{\mathrm{BM}}$, and the second
is dual to the inclusion of $P_g^{\mathcal A}$;
see \cite[Proposition~4.4 and proof of Theorem~1.1]
{BrownChanGalatiusPayneHopf}.
The first arrow is a left inverse to
$(\iota_{\mathcal A})_{g,g}$.
Thus
Lemmas~\ref{Lemma: Algebraic and tropical torelli comparison}
and~\ref{lem:bridgeless-graph-complex-to-Quillen}
give
\begin{equation*}\label{eq:dual-torelli-graph-composite}
T_g^\vee\circ\Theta_g
=
R_g\circ\rho_{g,g}\circ(\jmath_g)_*.
\end{equation*}

To compare the right hand side with $\beta_{\mathcal A,g}^\vee$,
we first show that it factors through the indecomposables of
graph homology. We begin by introducing a chain complex that
computes these indecomposables.

Let
\[
\varepsilon\colon
\operatorname{Gr}(C)\longrightarrow\mathbb R
\]
be the standard augmentation, and set
\[
K:=\ker(\varepsilon),
\qquad
L^\vee:=
\operatorname{Indec}(\operatorname{Gr}(C))
=
K/K^2.
\]
Since graph multiplication is disjoint union, a disjoint union
of two graphs with edges represents zero in $L^\vee$.
Let
\[
q_g\colon
\operatorname{Gr}_g(C)\longrightarrow L_g^\vee
\]
be the quotient map. By
\cite[Proposition~5.22]{BrownChanGalatiusPayneHopf},
the natural map
\[
\operatorname{Indec}
H_\bullet(\operatorname{Gr}(C))
\xrightarrow{\sim}
H_\bullet(L^\vee)
\]
is an isomorphism. We next identify the relevant homology of $L^\vee$ with
ordinary connected graph homology. Let $\mathsf G^{(g)}$ denote the homological graph complex
dual to $\mathsf{GC}_2^{(g)}$. The map
\[
\lambda_g\colon
\mathsf G_\bullet^{(g)}
\longrightarrow
(L_g^\vee)_{2g+\bullet},
\qquad
[\Gamma,\omega]
\longmapsto
[\Gamma,\omega]\bmod K^2,
\]
is a chain map. Indeed, deleting a bridge from a stable graph
produces a disjoint union of two graphs with edges, so this
term vanishes in $K/K^2$; the remaining differential is
edge contraction. By
\cite[Lemma~5.21 and Proposition~5.22]
{BrownChanGalatiusPayneHopf},
for $g\geq2$ this map induces an isomorphism
\[
(\lambda_g)_*\colon
H_0(\mathsf G^{(g)})
\xrightarrow{\sim}
H_{2g}(L_g^\vee).
\]
We therefore define
\begin{equation*}\label{eq:definition-Pi-g}
\Pi_g
:=
(\lambda_g)_*^{-1}\circ(q_g)_*
\colon
H_{2g}(\operatorname{Gr}_g(C))
\longrightarrow
H_0(\mathsf G^{(g)}).
\end{equation*}

We now verify that $R_g\circ\rho_{g,g}$ factors through
$\Pi_g$. The dual of the canonical Hopf quotient embeds
$P_g^{\mathcal A}$ as primitive elements in the cohomological
Quillen $E^1$-page. Since primitive elements vanish on products
of augmentation-zero homology classes, restriction to
$P_g^{\mathcal A}$ makes $R_g$ vanish on such products.
The map $\rho$ preserves multiplication and augmentation,
so $R_g\circ\rho_{g,g}$ also vanishes on products in the
augmentation ideal of graph homology. It therefore factors
through the indecomposables, and hence through $\Pi_g$.

Under the identifications above, the construction of
$\beta_{\mathcal A,g}$ in
\cite[equations~(81)--(82)]{BrownChanGalatiusPayneHopf}
identifies the induced map on ordinary connected graph
homology with $\beta_{\mathcal A,g}^\vee$. Consequently,
\begin{equation*}\label{eq:primitive-dual-comparison}
R_g\circ\rho_{g,g}
=
\beta_{\mathcal A,g}^\vee\circ\Pi_g.
\end{equation*}

It remains to understand what $\Pi_g$ does to a class
coming from the bridgeless graph complex.
Let
\[
u_g\colon
\bigl(\GKbarzeroBM\bigr)_{2g+\bullet}
\longrightarrow
\mathsf G_\bullet^{(g)}
\]
be the inclusion that preserves each oriented graph.
Both routes in
\[
q_g\circ\jmath_g
=
\lambda_g\circ u_g
\]
send a bridgeless graph to its class modulo $K^2$.
Passing to homology therefore gives
\begin{equation*}\label{eq:Pi-same-graph-comparison}
\Pi_g\circ(\jmath_g)_*
=
(u_g)_*.
\end{equation*}

We now return to the geometry. The bridgeless graph model
established above and the ordinary graph model of
\cite[Theorem~1.3]{ChanGalatiusPayne2021} give the respective
identifications
\[
H_{2g}(\GKbarzeroBM)
\cong
\widetilde H_{2g-1}(\Delta_g^{{bl}}),
\qquad
H_0(\mathsf G^{(g)})
\cong
\widetilde H_{2g-1}(\Delta_g).
\]
Under these isomorphisms, $(u_g)_*$ agrees with the map on reduced homology induced by the inclusion
\[
\Delta_g^{{bl}}\hookrightarrow\Delta_g.
\]

To relate this inclusion to extension by zero, use the common
compactification $\overline{\mathcal M}_g$. Under the standard
identifications of compactly supported cohomology with relative
cohomology, $i_{g,!}$ is the map
\[
H^{2g}\left(
\overline{\mathcal M}_g,
\partial\overline{\mathcal M}_g
\right)
\longrightarrow
H^{2g}\left(
\overline{\mathcal M}_g,
\Delta_{\mathrm{irr}}
\right)
\]
induced by
$\Delta_{\mathrm{irr}}\subseteq
\partial\overline{\mathcal M}_g$. Since $H^k(\Mb_g)$ is pure of weight $k$, its weight zero
part vanishes for $k>0$. Thus the connecting homomorphisms
in the long exact sequences associated to the pairs
$(\Mb_g,\partial\Mb_g)$ and
$(\Mb_g,\Delta_{\mathrm{irr}})$ induce isomorphisms
\[
\begin{aligned}
W_0H^{2g-1}(\partial\Mb_g)
&\xrightarrow{\sim}
W_0H_c^{2g}(\mathcal M_g),\\
W_0H^{2g-1}(\Delta_{\mathrm{irr}})
&\xrightarrow{\sim}
W_0H_c^{2g}(\mathcal M_g^{ct}).
\end{aligned}
\]
These connecting homomorphisms commute with restriction
from $\partial\overline{\mathcal M}_g$ to
$\Delta_{\mathrm{irr}}$. Thus, after identifying the
weight zero cohomology of these boundaries with the
cohomology of their dual complexes, $i_{g,!}$ agrees with
restriction
\[
\widetilde H^{2g-1}(\Delta_g)
\longrightarrow
\widetilde H^{2g-1}(\Delta_g^{{bl}}).
\]

Dualizing gives the same inclusion-induced homology map
described above. By the definitions of $\Theta_g$ and
$\gamma_{\mathcal M,g}$, we therefore obtain
\begin{equation*}\label{eq:u-extension-comparison}
(u_g)_*
=
\gamma_{\mathcal M,g}^\vee
\circ(i_{g,!})^\vee
\circ\Theta_g.
\end{equation*}

Combining these identities, we obtain
\[
\begin{aligned}
T_g^\vee\circ\Theta_g
&=
R_g\circ\rho_{g,g}\circ(\jmath_g)_*
\\
&=
\beta_{\mathcal A,g}^\vee
\circ\Pi_g\circ(\jmath_g)_*
\\
&=
\beta_{\mathcal A,g}^\vee\circ(u_g)_*
\\
&=
\beta_{\mathcal A,g}^\vee
\circ\gamma_{\mathcal M,g}^\vee
\circ(i_{g,!})^\vee
\circ\Theta_g
\\
&=
\eta_{\mathcal A,g}^\vee
\circ(i_{g,!})^\vee
\circ\Theta_g.
\end{aligned}
\]
Since $\Theta_g$ is an isomorphism, cancelling it and
dualizing gives
\[
(\tau_g^{ct})^*|_{P_g^{\mathcal A}}
=
i_{g,!}\circ\eta_{\mathcal A,g},
\]
as required.
\end{proof}

We now deduce Theorem~\ref{thm:intro-motivic-comparison}
from Theorem~\ref{thm:compact type-torelli-comparison}.
We first specify the two motivic families appearing in its
statement. For $k\geq1$, let
\[
\omega^{4k+1}
\in
W_0H_c^{4k+2}(\mathcal A_{2k+1};\mathbb R)
\]
denote the primitive class constructed in
\cite{BrownChanGalatiusPayneHopf}.
This class is denoted $[\omega_c^{4k+1}]$ in
\cite{BrownPanzerPortner2026}. Let
\[
\L_{\mathcal A}\subseteq
\operatorname{Prim}(D_{\mathcal A})
\]
be the real Lie subalgebra generated by
$\omega^5,\omega^9,\omega^{13},\ldots$.

Let $\mathfrak g^{\mathfrak m}$ denote the motivic Lie algebra
defined in
\cite[Section~3.1.4]{BrownPanzerPortner2026}.
It is a free graded Lie algebra with one generator in each
odd degree at least $3$. Set
\[
\mathfrak g_{\mathbb R}^{\mathfrak m}
:=
\mathfrak g^{\mathfrak m}\otimes_{\mathbb Q}\mathbb R.
\]
By \cite[Theorem~3]{BrownPanzerPortner2026}, the Lie subalgebra
of $H^0(\mathsf{GC}_2)\otimes_{\mathbb Q}\mathbb R$ generated
by the classes of the canonical graph cocycles is canonically
isomorphic to $\mathfrak g_{\mathbb R}^{\mathfrak m}$.
Using this realization, we regard
\[
\mathfrak g_{\mathbb R}^{\mathfrak m}
\subseteq
H^0(\mathsf{GC}_2)\otimes_{\mathbb Q}\mathbb R
\]
and define
\[
\L_{\mathcal M}
:=
\gamma_{\mathcal M}
\bigl(\mathfrak g_{\mathbb R}^{\mathfrak m}\bigr)
\subseteq D_{\mathcal M}.
\]

Finally, write
\[
i_!:=\bigoplus_{g\geq2}i_{g,!},
\qquad
(\tau^{ct})^*
:=
\bigoplus_{g\geq2}(\tau_g^{ct})^*.
\]

\begin{proof}[Proof of Theorem~\ref{thm:intro-motivic-comparison}]

By \cite[Theorems~3 and~24]{BrownPanzerPortner2026},
the map $\beta_{\mathcal A}$ identifies $\L_{\mathcal A}$
with the Lie subalgebra generated by the canonical graph
cohomology classes. Under the realization chosen above,
this gives an isomorphism
\[
\beta_{\mathcal A}|_{\L_{\mathcal A}}\colon
\L_{\mathcal A}
\xrightarrow{\sim}
\mathfrak g_{\mathbb R}^{\mathfrak m}.
\]

Since
$\eta_{\mathcal A}
=\gamma_{\mathcal M}\circ\beta_{\mathcal A}$,
it follows that $\eta_{\mathcal A}$ restricts to an
isomorphism
\[
\eta_{\mathcal A}|_{\L_{\mathcal A}}\colon
\L_{\mathcal A}
\xrightarrow{\sim}
\L_{\mathcal M}.
\]
Every element of $\L_{\mathcal A}$ is primitive, so
Theorem~\ref{thm:compact type-torelli-comparison},
applied genus by genus, gives
\[
(\tau^{ct})^*|_{\L_{\mathcal A}}
=
i_!\circ\bigl(\eta_{\mathcal A}|_{\L_{\mathcal A}}\bigr).
\]
Taking images yields
\[
(\tau^{ct})^*(\L_{\mathcal A})
=
i_!(\L_{\mathcal M}),
\]
as required. Moreover, $(\tau^{ct})^*|_{\L_{\mathcal A}}$ is injective.
Indeed, $\eta_{\mathcal A}|_{\L_{\mathcal A}}$ is an
isomorphism, while $i_!$ is injective on $D_{\mathcal M}$
by Theorem~\ref{thm: i: Mg to Mgct}. Their composite
$(\tau^{ct})^*|_{\L_{\mathcal A}}$ is therefore injective.

\end{proof}


\bibliographystyle{alpha}
\bibliography{references}

\end{document}